\documentclass{article}
\usepackage[utf8]{inputenc}
\usepackage{setspace}

\usepackage{geometry}
\usepackage{graphicx,wrapfig,lipsum,booktabs}
\graphicspath{ {images/} }
\usepackage{amsmath}
\usepackage{amsthm}
\usepackage{tikz}
\usepackage{caption}

\input{insbox}

\newtheorem{Theorem}{Theorem}[]
\newtheorem{Claim}[Theorem]{Claim}
\newtheorem{Corollary}[Theorem]{Corollary}
\newtheorem{Conjecture}{Conjecture}
\newtheorem*{Remark}{Remark}
\newtheorem{Definition}[Theorem]{Definition}

\def\perm{{\operatorname{perm}}}

\tikzstyle{every node}=[circle, draw, fill=black!50, inner sep=0pt, minimum width=4pt]
\tikzstyle{text}=[circle, draw, fill=blue!50, inner sep=0pt, minimum width=4pt]

\title{Maximum determinant and permanent of sparse 0-1 matrices}
\author{
	Igor Araujo \footnote{Department of Mathematics, University of Illinois at Urbana-Champaign. E-mail: \texttt{igoraa2@illinois.edu}.} \qquad 
	J\'ozsef Balogh \footnote{Department of Mathematics, University of Illinois at Urbana-Champaign, Urbana, Illinois 61801, USA, and Moscow Institute of Physics and Technology, Russian Federation. E-mail: \texttt{jobal@illinois.edu}. Research supported by NSF RTG Grant DMS-1937241, NSF Grant DMS-1764123, Arnold O. Beckman Research Award (UIUC Campus Research Board RB 18132), the Langan Scholar Fund (UIUC), and the Simons Fellowship.} \qquad
	Yuzhou Wang \footnote{Department of Mathematics, University of Illinois at Urbana-Champaign. Member of Illinois Combinatorics Lab for Undergraduate Experiences. E-mail: \texttt{yuzhouw3@illinois.edu}.}
	}
\date{\today}

\begin{document}
\maketitle

\begin{abstract}
	We prove that the maximum determinant of an $n \times n $ matrix, with entries in $\{0,1\}$ and at most $n+k$ non-zero entries, is at most $2^{k/3}$, which is best possible when $k$ is a multiple of 3. This result solves a conjecture of Bruhn and Rautenbach. We also obtain an upper bound on the number of perfect matchings in $C_4$-free bipartite graphs based on the number of edges, which, in the sparse case, improves on the classical Bregman's inequality for permanents. This bound is tight, as equality is achieved by the graph formed by vertex disjoint union of 6-vertex cycles.
\end{abstract}

\section{Introduction}

Many upper bounds on determinants have been given in the literature, in particular to matrices with all entries 0 or 1. Assuming the matrix $A\in \{0,1\}^{n \times n}$ has at most $2n$ non-zero entries, then the classical Hadamard's inequality \cite{Hadamard}
$ \det(A) \leq \prod_{i=1}^n \left( \sum_{j=1}^n a_{i,j}^2 \right)^{1/2},$
together with arithmetic and geometric mean inequality, imply $\det(A) \leq 2^{n/2}$. This bound was improved to $\det(A)\leq 2\left(2- \frac{2}{n-1}\right)^{\frac{1}{2}(n-1)}$ by Ryser \cite{Ryser}, who extended it as follows:
$$ \det(A) \leq k \left(k- \frac{k(k-1)}{n-1}\right)^{\frac{1}{2}(n-1)} \text{ when } A \in \{0,1\}^{n \times n} \text{ has } kn \text{ ones, for } 1\leq k \leq \frac{n+1}{2}.$$

While both classical bounds above are best-possible in their general formulation, they are not best possible for sparse matrices. Ryser's inequality holds with equality only when we have at least $n\sqrt{n}$ non-zero entries (see \cite{Bruhn} or \cite{Ryser} for a more detailed discussion on when equality holds for Ryser's inequality). Aiming to obtain better bounds for sparse combinatorial matrices, Bruhn and Rautenbach \cite{Bruhn} proved the following.

\begin{Theorem}[Bruhn, Rautenbach]
	If $A \in \{0, 1\}^{n\times n}$ has at most $2n$ non-zero entries, then $| det(A)| \leq 2^{n/6} \cdot 3^{n/6}$.
\end{Theorem}

They also conjectured that the determinant of $A$ is at most $2^{n/3}$ in this case (see Conjecture 4 in \cite{Bruhn}). Advancing towards this conjecture, Shitov \cite{Shitov} generalized its formulation and used induction to give a short and elegant proof of the following result.

\begin{Theorem}[Shitov]
	If $A \in \{0, 1\}^{n\times n}$ has at most $n+k$ non-zero entries, then $| det(A)| \leq 3^{k/4}$.
\end{Theorem}

\noindent The conjectured maximum value comes from matrices of the form $A = \text{diag}(C,\dots ,C)$, where 
$C = \begin{pmatrix} 1 & 1 & 0 
	\cr 0 & 1 & 1 
	\cr 1 & 0 & 1 
\end{pmatrix}$, which has determinant of $2^{n/3}$ with $n=k$. The main contribution of the present paper is a proof of the optimal result.

\begin{Theorem} \label{thm:det}
Every matrix $A \in \{0,1\}^{n \times n}$ containing at most $n+k$ non-zero entries has determinant at most $\alpha^{k}$, where $\alpha = 2^{1/3}$.
\end{Theorem}

In particular, this bound is best possible only when $k$ is a multiple of 3, and $k \leq n$. We emphasize that Theorem~\ref{thm:det} resolves the above conjecture of Bruhn and Rautenbach.

\begin{Corollary} \label{cor:det}
	If $A \in \{0, 1\}^{n\times n}$ has at most $2n$ non-zero entries, then $| det(A)| \leq 2^{n/3}$.
\end{Corollary}

We can consider the same question for more 1's. Bruhn and Rautenbach \cite{Bruhn} noted that the point-line incidence matrix of the Fano plane has determinant 24. It gives a lower bound of $24^{n/7}\approx 1.5746^n$ for the maximum determinant of matrices with at most $3n$ ones, the authors of \cite{Bruhn} conjecture this to be the best possible. Scheinerman \cite{Scheinerman} showed $\det(A) \leq c(k)^n$ for some constant $c(k)$ depending only on the integer $k$ for all matrices $A$ with at most $kn$ ones. For $k=3$, Scheinerman bound is $\det(A) \leq 24^{n/6} \approx 1.6984^n$. Theorem \ref{thm:det} improves this bound. 

\begin{Corollary} \label{cor:det3n}
	If $A \in \{0, 1\}^{n\times n}$ has at most $3n$ non-zero entries, then $| det(A)| \leq 2^{2n/3} \approx 1.5874^n$.
\end{Corollary}  

Our proof of Theorem \ref{thm:det} extends some of the ideas of Shitov \cite{Shitov}. With a careful analysis, we push down the bound to the optimal value of $2^{k/3}$. We identify a matrix $A$ with the graph whose bi-adjacency matrix is $A$. In other words, when $G$ is a balanced bipartite graph, the determinant of $G$, denoted by $\det(G)$, is the absolute value of the determinant of the bi-adjacency matrix of $G$. We then aim to prove $\det(G) \leq 2^{k/3}$ for all balanced bipartite graphs $G$ with $2n$ vertices and at most $n+k$ edges. The proof will be by induction on $n+k$: given a graph $G$, we will assume the inequality $\det(G') \leq \alpha^{e(G')-v(G')/2}$ for all proper balanced bipartite subgraphs $G'$ of $G$, where $v(G')$ and $e(G')$ denote the number of vertices and edges of $G'$, respectively. A bipartite graph is \textit{balanced} if both parts have the same number of vertices.  

For the sake of completeness of the induction argument, we highlight that the result holds for every $n>0$ when $k=0$. For bigger values of $n$ and $k$, we break the proof into several cases. 

Since in most cases we only make use of linearity and cofactor expansion along lines, we get as a byproduct an upper bound for permanents instead of only determinants. In what follows, $\perm(G)$ stands for the permanent of the bi-adjacency matrix of $G$ or, equivalently, the number of perfect matchings in $G$. 

We highlight that the first part of the next result might be of independent interest. Namely, that $\perm (G) \leq 2^{k/3}$ for all $C_{4}$-free bipartite graphs $G$ with $2n$ vertices and $n+k$ edges, we state it in Theorem \ref{thm:maxdeg3,4,5}. We have additional stronger statements for connected graphs with a minimum degree of at most 2 that make the proof of the induction step simpler. This stronger statement tells us that the inequality holds with an extra multiplicative factor whenever we avoid components inducing $K_2$ and $C_6$. To get an even stronger result and make our proof to work, we rule out one further graph, namely the graph $J$ below.
\begin{figure}[h!]
	\centering
	\caption{The graph $J$.}
\begin{tikzpicture}[thick,scale=0.7]%
	\draw (0:2) -- (0.7,0) node {}
	(180:2) -- (-0.7,0) node {}
	(-0.7,0) -- (0.7,0);
	\foreach \x in {0,60,120,180,240,300} 
	{\draw (\x:2) node {} -- (\x+60:2) node {};}
\end{tikzpicture} 
\end{figure}
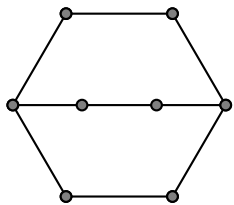

From now on, we use $J$ only to refer to that graph. For the sake of brevity, we write $\alpha=2^{1/3}$.

\begin{Theorem} \label{thm:maxdeg3,4,5}${}$
	\begin{itemize}
		\item[(a)] Let $G$ be a $C_{4}$-free balanced bipartite graph with $2n$ vertices, $n+k$ edges. Then $\perm (G) \leq \alpha^{k}$. 
		
		\item[(b)] Let $H$ be a connected $C_{4}$-free balanced bipartite graph with $2n$ vertices, $n+k$ edges, $\delta(H) \leq 2$, and $\Delta(G) \leq 3$. Assume further that $H$ is not isomorphic to $K_2$, $C_6$, or $J$. Then $\perm (H) \leq c_1 \cdot \alpha^{k}$, where $c_1 = \alpha^{-2}+\alpha^{-7} \leq (3\alpha^{-4})^{-1}$ is a constant.
		
		\item[(c)] Let $H$ be a connected $C_{4}$-free balanced bipartite graph with $2n$ vertices, $n+k$ edges, and $\delta(H) \leq 2$. Assume further that $H$ is not isomorphic to $K_2$, $C_6$, or $J$. Then $\perm (H) \leq c_2 \cdot \alpha^{k}$, where $c_2 = \alpha^{-3} + \alpha^{-4} \leq (\alpha^{-1}+\alpha^{-5})^{-1}$ is a constant.
	\end{itemize}
\end{Theorem}

\begin{Remark} \normalfont
	We note that the values $c_1= \alpha^{-2}+\alpha^{-7} \approx 0.828$ and $c_2 = \alpha^{-3} + \alpha^{-4} \approx 0.897$ that our proof gives in the above statement are not tight. Furthermore, parts (b) and (c) of Theorem \ref{thm:maxdeg3,4,5} with $c_1=(3\alpha^{-4})^{-1} \approx 0.840$ and $c_2=(\alpha^{-1}+\alpha^{-5})^{-1} \approx 0.902$ would be sufficient to conclude part (a). 
\end{Remark}

For sparse graphs, Theorem \ref{thm:maxdeg3,4,5} improves on the classic Bregman's inequality \cite{Bregman} for permanents. For a $C_4$-free $d$-regular graph Theorem \ref{thm:maxdeg3,4,5} gives an upper bound of $2^{(d-1)n/3}$, which is smaller than the bound $(d!)^{n/d}$ from Bregman's inequality when $d \leq 5$.

We prove Theorem \ref{thm:maxdeg3,4,5} in the Sections \ref{sec:overview}--\ref{sec:proof(a)}. The proof of Theorem \ref{thm:det} is analogous, with the small difference that we have to take into consideration when the graph $G$ has $C_4$ as a subgraph, where we use linearity or the fact the determinant is preserved after subtracting a line from another line instead of using cofactor expansion. As the majority of the proofs are similar, we only present a sketch proof for Theorem \ref{thm:det} in Section \ref{sec:detproof} with handling the additional cases where the proof of Theorem \ref{thm:maxdeg3,4,5} has to be supplemented. 

\section{Proof of Theorem \ref{thm:maxdeg3,4,5}: Overview} \label{sec:overview}

Throughout the proof we identify a matrix $A$ with the graph whose bi-adjacency matrix is $A$. Therefore, we label the lines of a matrix by the vertex set of a graph. Let $G$ be a balanced bipartite graph. The cofactor expansion along a vertex $u$ adjacent to $v_1, \dots, v_t$ for permanents implies $$\perm(G) = \sum_{i=1}^t \perm(G-\{u,v_i\}),$$
where $G-\{u,v_i\}$ is the graph obtained after deleting the vertices $u$ and $v_i$ from $G$. 

We introduce an auxiliary function 
$$f(G) := \alpha^{-e(G)+\frac{1}{2}v(G)} \cdot \perm(G),$$ 
where $e(G)$ and $v(G)$ are the number of the edges and vertices of $G$, respectively. In this notation, the cofactor expansion implies 
$$ f(G) = \sum_{i=1}^t  \alpha^{2-d(u)-d(v_i)} \cdot  f(G-\{u,v_i\}).$$

We can think of $f(G)$ as the normalized number of perfect matchings of $G$. Hence, we want to prove that $f(G)\leq 1$ for all $C_{4}$-free balanced bipartite graphs $G$, and $f(H)\leq c_1$ for connected graphs $H$ with minimum degree at most 2 and maximum degree at most 3 that are not isomorphic to $K_2$, $C_6$ or $J$, and $f(H)\leq c_2$ for connected graphs $H$ with minimum degree at most 2 that are not isomorphic to $K_2$, $C_6$ or $J$. 

We will prove these statements with a simultaneous induction. In Section \ref{sec:proof(b,c)}, we present the proofs of (b) and (c) for the graph $H$, assuming the results of (a), (b), and (c) hold for all proper balanced subgraphs of $H$. 
For the proof of (b), since $H$ is connected, we do not have isolated vertices. If $H \neq K_2$ has vertex with degree 1, we expand the permanent along this vertex and conclude that $f(H)\leq \alpha^{-1}<c_1$. We thus assume $\delta(H)=2$ and $\Delta(H)\leq 3$.
If $H$ has a path with three consecutive vertices of degree 2, then we either proceed as in Claim \ref{claim:22-23} to obtain $f(H)\leq \alpha^{-3}+\alpha^{-5} < c_1$ or $H$ has to be a cycle, in which case we are done since $H \neq C_4, C_6$ and $f(C_{2n}) < c_1$ for $n\geq 4$. 

If $H$ has a vertex $x$ of degree 2 adjacent to vertices $y_1$ and $y_2$ of degree 3, then we proceed as in Claim \ref{claim:2-3,3} to obtain 
\begin{equation} \label{eq:typeI}
	f(H) \leq \frac{1}{2}f(H-\{x,y_1\})+\frac{1}{2}f(H-\{x,y_2\}).
\end{equation}

Assuming $H-\{x,y_i\}$ is connected and not isomorphic to $K_2$, $C_6$, or $J$, the bound $f(H) \leq c_1$ follows from $f(H-\{x,y_i\}) \leq c_1$ for $i=1,2$. Otherwise, we will see, in Claim \ref{claim:Type1Forb}, that $H-\{x,y_i\}$ cannot be isomorphic to any of $K_2, C_6, J$; and, in Claim \ref{claim:Type1Disc}, that if $H-\{x,y_i\}$ is not connected, then $f(H) \leq c_1$. 

If there is no path with three consecutive vertices of degree 2 in $H$ and $H$ has no vertex of degree 2 connected to two vertices of degree 3, then any vertex $x$ of degree 2 in $H$ is as in Figure \ref{fig:type2}.
We then proceed as in Claim \ref{claim:23-23} to obtain 
\begin{equation} \label{eq:typeII}
	f(H) \leq \alpha^{-2}f(H-\{x,y_1\})+\alpha^{-5}f(H-\{x,x_1,y_1,y_2\}).    
\end{equation}

We will see, in Claims \ref{claim:Type2Forb2} through \ref{claim:Type2Disc2}, that $H-\{x,y_1\}$ and $H-\{x,x_1,y_1,y_2\}$ cannot be isomorphic to $K_2$, $C_6$, or $J$; and that if $H-\{x,y_1\}$ or $H-\{x,x_1,y_1,y_2\}$ are not connected, then $f(H) \leq c_1$. Otherwise, we use the induction hypothesis to conclude $f(H)\leq c_1$ from $f(H-\{x,y_1\}) \leq c_1$ and $f(H-\{x,x_1,y_1,y_2\}) \leq c_1$. 

\begin{Definition} \label{def:type12}
	We call $x \in V(H)$ a \textbf{Type I} vertex if $x$ is of degree 2 and its neighbors are of degree 3. We call $x \in V(H)$ a \textbf{Type II} vertex if $x$ is as in Figure \ref{fig:type2}, i.e., $x$ is adjacent to $y_1$ of degree 2 and $y_2$ of degree 3, while $y_1$ is also adjacent to a vertex $x_1$ of degree 3. Therefore, when we say ``\emph{Type I} deletion" or ``\emph{Type I} expansion" we mean expanding the permanent along a Type I vertex as in \eqref{eq:typeI} above, and similarly for ``Type II deletion"  or ``\emph{Type II} expansion" as in \eqref{eq:typeII}.
	Further, whenever we fix a Type I vertex $x$, the variables $y_1$ and $y_2$ stand for the neighbors of $x$. For a Type II vertex $x$, the variables $y_1$ and $y_2$ stand for the neighbors of $x$, and $x_1 \neq x$ is the only other neighbor of $y_1$.
\end{Definition}

\begin{figure}[h!]
	\begin{minipage}{0.50\textwidth}
		\centering
		\caption{\emph{Type I} vertex.}
		\begin{tikzpicture}[thick,scale=0.7] 
			\draw (0,1) node [label=left:$(2) x$] {} -- (2,2) node [label=right:$y_1 (3)$] {} 
			(0,1) -- (2,0) node [label=right:$y_2 (3)$] {};
		\end{tikzpicture}
	\end{minipage}
	\begin{minipage}{0.40\textwidth} 
		\centering
		\caption{\emph{Type II} vertex.}
		\begin{tikzpicture}[thick,scale=0.7]
			\draw (0,2) node [label=left:$(2) x$] {} -- (2,2) node [label=right:$y_1 (2)$] {}   
			(0,2) -- (2,0) node [label=right:$y_2 (3)$] {}  
			(0,0) node [label=left:$(3) x_1$] {} -- (2,2); 
		\end{tikzpicture} \label{fig:type2}
	\end{minipage}
\end{figure}

We can think of Definition \ref{def:type12} as a way to classify degree 2 vertices in $H$. Each vertex of degree 2 (that is not contained in a path with three consecutive vertices of degree 2) in $H$ is either \emph{Type I} or \emph{Type II}. Depending on whether there is a Type I vertex in $H$ or all vertices of degree 2 are Type II vertices we will use the expansion \eqref{eq:typeI} or \eqref{eq:typeII}, respectively. 
For the proof of (c), we proceed similarly as the one for (b), we deal with the proof of (c) in Section \ref{sec:proof(c)}.

We prove (a) in Section \ref{sec:proof(a)}. First, in Section \ref{sec:deg2}, we use (b) and (c) to reduce the proof of (a) to fewer cases. Next, we deal with the cases when $G$ is disconnected, or $\Delta(G)\geq 6$, or $\delta(G)\geq 4$ in Subsections \ref{sec:disc}, \ref{sec:deg6}, and \ref{sec:mindeg4}, by using (a) for some proper subgraphs of $G$.

After that, we can assume $G$ is connected, $\delta(G)=3$ and $\Delta(G) \leq 5$. We thus use (a) and (b) for proper subgraphs of $G$ to deal with the case $\Delta(G) = 3$ in Section \ref{3reg:firststep}, and use (a) and (c) for proper subgraphs of $G$ to deal with the cases $\Delta(G) = 4$ and $\Delta(G) = 5$ in Sections \ref{maxdeg4:firststep} and \ref{maxdeg5:firststep}, respectively. In all cases, we use cofactor expansion or linearity to bound the permanent by a sum of permanents of subgraphs with minimum degree 2 in a way that the induction hypothesis implies $f(G)\leq 1$.

For example, when $G$ is 3-regular, the cofactor expansion along a vertex $u$ adjacent to $v_1$, $v_2$ and $v_3$ is equivalent to
$$f(G) = \alpha^{-4} \cdot \left( \sum_{i=1}^3 f(G-\{u,v_i\}) \right) .$$

\noindent Our goal is to use the induction hypothesis of part (b) for $H=G-\{u,v_i\}$ to get $f(H) \leq c_1$, therefore concluding
$$f(G) = \alpha^{-4} \cdot \left( \sum_{i=1}^3 f(G-\{u,v_i\}) \right) \leq 
\alpha^{-4} \cdot \left( \sum_{i=1}^3 \frac{\alpha^4}{3} \right)  = 1 . $$  
However, we can use part (b) only when $H$ is connected and $H \neq K_2, C_6, J$. We will see, in Claim \ref{claim:1stStepForb}, that $G-\{u,v_i\}$ cannot be isomorphic to $K_2$, $C_6$ or $J$. Finally, in Claim \ref{claim:1stStepDisc}, that $f(G) \leq 1$ when $G-\{u,v_i\}$ is not connected follows from (a) for proper subgraphs of $G$. 

\section{Proofs of Theorem \ref{thm:maxdeg3,4,5} (b) and (c)} \label{sec:proof(b,c)}

The proofs are by induction on $n+k$. When $n\leq 1$ or $k=0$, there is no connected bipartite graph $H\neq K_2$. For $n=2$, the only connected $C_4$-free graph is a path with $n+k=3$ edges, which has permanent $1\le c_1\alpha <  c_2\alpha$. Then the induction hypothesis for both (b) and (c) are true when $n+k \leq 3$.

Further, as $H$ is connected, it has no isolated vertices. If $H \neq K_2$ has a vertex $v$ of degree 1, then the neighbor $w$ of $v$ has degree at least 2. Expanding on the line of vertex $v$, we get $\perm(H) = \perm(H-\{v,w\})$ and $f(H) \leq \alpha^{-1} \leq c_1$. If $H$ is a cycle, we have $\perm(C_{2n})=2$ and then $f(C_{2n})=2\alpha^{-n} \leq \alpha^{-1}$ since $H \neq C_4, C_6$ and $n \geq 4$. From now we can assume that the minimum degree of $H$ is 2 and the maximum degree is at least 3.

If $H$ is not a cycle and has a path with three consecutive vertices of degree 2, then it is sufficient to use part (a) of the induction hypothesis to obtain 
$$f(H)\leq \alpha^{-3}+\alpha^{-5}< 0.815 < 0.828 < c_1,$$
using the following Claim. 

\begin{Claim} \label{claim:22-23}
	Let $u$ and $v_1$ be adjacent vertices with degree 2. Further, assume that $u$ is adjacent to $v_{2}$, $v_{1}$ is adjacent to $u_{1}$, and $v_{2}$ is not adjacent to $u_{1}$. If $d(u_{1})=2$ and $d(v_{2}) \geq 3$, then $\perm(H) \leq \alpha^{k-3} + \alpha^{k-5} \leq c_1 \cdot \alpha^{k}$.
\end{Claim}

\begin{proof}
	Assume $u_1$ is adjacent to a vertex $v_{3}$, and $v_3$ is adjacent to a vertex $u_2$ ($\neq u, u_1$). Then the bi-adjacency matrix of $H$ is
	\begin{center}
		\begin{tikzpicture}[thick,scale=0.8]
			\draw[color=white, use as bounding box] (-11,-0.7) rectangle (4,2.7);
			\draw (0,2) node [label=left:$(2) u$] {} -- (2,2) node [label=right:$v_1 (2)$] {}
			(0,2) -- (2,0) node [label=right:$v_2 (\geq 3)$] {}
			(0,0) node [label=left:$(2) u_1$] {} -- (2,2); 
			\node[draw=none,fill=none]  at (-7,1) {$A = \bordermatrix{&v_1&v_2&v_3& \ldots
					\cr u & 1 & 1 & 0 & 0&\ldots & 0
					\cr u_1&1 & 0 & 1 &0&\ldots & 0
					\cr u_2&0 & x &1 &a_{34}& \ldots & a_{3n}
					\cr \vdots & \vdots & \vdots & \vdots &\vdots & \ddots & \vdots
					\cr & 0 & a_{n2} & a_{n3} &a_{n4}&\ldots & a_{nn}}$};
		\end{tikzpicture}
	\end{center}

	\noindent By expanding by the line of $u$ and then $u_1$ and $v_1$, respectively we have
	$$\perm(H) \leq \perm(H-\{u,v_1\})+\perm(H-\{u,v_2\}) \leq \perm(H-\{u,v_1,u_1,v_3\})+\perm(H-\{u,v_2,v_1,u_1\}) .$$
	
	If $d(v_3) \geq 3$, then we conclude $\perm(H) \leq \alpha^{k-4}+\alpha^{k-4}<\alpha^{k-3}+\alpha^{k-5}$. Assume, then, that $d(v_3)=2$. We expand $H-\{u,v_2,v_1,u_1\}$ along the line of $v_3$. If $d(u_2)\geq 3$, or $d(u_2) = 2$ and $u_2$ is not adjacent to $v_2$, then 
	\begin{align*}
		\perm(H) & \leq \perm(H-\{u,v_1,u_1,v_3\})+\perm(H-\{u,v_2,v_1,u_1\}) \\
		& \leq \perm(H-\{u,v_1,u_1,v_3\})+\perm(H-\{u,v_2,v_1,u_1, v_3,u_2\}) \leq \alpha^{k-3} + \alpha^{k-5}.    
	\end{align*}
	
	\noindent Finally, if $d(u_2) = 2$ and $u_2$ is adjacent to $v_2$, then we expand $H-\{u,v_1,u_1,v_3\}$ along the line of $u_2$ to get
	\begin{align*}
		\perm(H) & \leq \perm(H-\{u,v_1,u_1,v_3\})+\perm(H-\{u,v_2,v_1,u_1\}) \\
		& \leq \perm(H-\{u,v_1,u_1,v_3,u_2,v_2\})+\perm(H-\{u,v_2,v_1,u_1, v_3,u_2\}) \leq \alpha^{k-4} + \alpha^{k-4}. \tag*{\qedhere}
	\end{align*}
\end{proof}

We thus assume that $H$ has no path with three consecutive vertices of degree 2. This implies that every vertex of degree 2 has a neighbor of degree at least 3. Let $y_{1}$ and $y_{2}$ be the neighbors of a degree 2 vertex $x$. We analyze all possible degree combinations of $y_{1}$ and $y_{2}$. 
Since we use extensively the same proof method in what follows, we explain in detail how to obtain the bounds in Claims \ref{claim:2-3,3} and \ref{claim:23-23}. We will expand the permanent along the vertices $x$ of degree 2 that are either \emph{Type I} or \emph{Type II} vertices and use the induction hypothesis of parts (b) and (c). Hence, we check the hypothesis of being connected and not isomorphic to $K_2$, $C_6$, or $J$ in Sections \ref{sec:type1} and \ref{sec:type2}. This will be sufficient to conclude part (b). To conclude part (c), we note that $x$ could be neither of Type I nor Type II, i.e. a neighbor of $x$ can have a degree larger than 3, and we handle these remaining cases in Section \ref{sec:proof(c)}. We first assume $H$ has a Type I vertex $x$.

\begin{Claim} \label{claim:2-3,3}
	Let $x$ be a vertex with degree 2 and neighbors $y_{1}$ and $y_{2}$. If $d(y_{1}) = 3$ and $d(y_{2}) = 3$, then 
	$$ f(H) \leq \frac{1}{2}f(H-\{x,y_1\})+\frac{1}{2}f(H-\{x,y_2\}).$$ 
	Futher, if both $H-\{x,y_1\}$ and $H-\{x,y_2\}$ are connected, not isomorphic to $K_2$, $C_6$, or $J$, then we have $f(H) \leq c_2$. Moreover, if $\Delta(H) \leq 3$, then $f(H) \leq c_1$.
\end{Claim}

\begin{proof}
	In this case, the bi-adjacency matrix of $H$ is
	
	\begin{center}
	\begin{tikzpicture}[thick,scale=0.8]
		\draw[color=white, use as bounding box] (-11,-1.7) rectangle (4,1.7);
		\draw (0,0) node [label=left:$(2) x$] {} -- (2,1) node [label=right:$y_1 (3)$] {}
		(0,0) -- (2,-1) node [label=right:$y_2 (3)$] {};
		\node[draw=none,fill=none]  at (-7,0) {$A= \bordermatrix{&y_1&y_2& \ldots 
			\cr x & 1 & 1 & 0 &\ldots & 0
			\cr &a_{21}&a_{22} & a_{23}&\ldots & a_{2n}
			\cr \vdots & \vdots & \vdots & \vdots & \ddots & \vdots
			\cr & a_{n1} & a_{n2} & a_{n3} &\ldots & a_{nn}} $};
	\end{tikzpicture}   
	\end{center}
	
	The first inequality follows from the cofactor expansion along the line of $x$. The submatrix obtained by deleting row $x$ and column $y_{i}$ of $A$ contains at most $n+k-d(x)-d(y_{i})+1$ non-zero entries. Since the $(x,y_i)$-minor corresponds to the permanent of an $(n-1) \times (n-1)$ matrix, if $H-\{x,y_i\}$ is connected, and not isomorphic to $K_2$, $C_6$, or $J$, then, as $H-\{x,y_i\}$ has minimum degree at most 2, we can use part (c) of the induction hypothesis to obtain that
	$\perm(H-\{x,y_i\}) \leq c_2 \cdot \alpha^{k-d(x)-d(y_{i})+2}$.
	We then conclude
	$$ \perm(H) \leq \perm(H-\{x,y_1\})+\perm(H-\{x,y_2\}) \leq c_2 \cdot \alpha^{k-d(x)-d(y_{1})+2}+ c_2 \cdot \alpha^{k-d(x)-d(y_{2})+2} \leq c_2 \cdot \alpha^{k},$$
	which is equivalent to $f(H) \leq c_2$. Similarly, if $\Delta(H) \leq 3$, then $f(H) \leq c_1$ by the induction hypothesis of part (b).
\end{proof} 

To conclude part (b), it remains to check cases with at least one of $y_{1}$ or $y_{2}$ having degree 2. Without loss of generality, let $d(y_{1}) = 2$, then $d(y_2)=3$.  Besides $x$, let $y_{1}$ be adjacent to $x_{1}$. We can assume $y_2$ is not adjacent to $x_1$, otherwise $x$ is contained in a $C_4$. As there is no path of three vertices each of degree  2,  we have $d(x_{1})=d(y_{2})=3$. 

\begin{Claim} \label{claim:23-23}
	Let $x$ and $y_1$ be adjacent vertices with degree 2. Further, assume that $x$ is adjacent to $y_{2}$, $y_{1}$ is adjacent to $x_{1}$, and $y_{2}$ is not adjacent to $x_{1}$. If $d(x_{1})=d(y_{2})=3$, then 
	$$f(H) \leq \alpha^{-2}f(H-\{x,y_1\})+\alpha^{-5}f(H-\{x,x_1,y_1,y_2\}). $$
	Futher, if both $H-\{x,y_1\}$ and $H-\{x, x_1, y_1, y_2\}$ are connected, not isomorphic to $K_2$, $C_6$, or $J$, and have minimum degree at most 2, then we have $f(H) < c_2$. Moreover, if $\Delta(H) \leq 3$, then $f(H) < c_1$.
\end{Claim}

\begin{proof}
	In this case 	
	\begin{center}
		\begin{tikzpicture}[thick,scale=0.8]
			\draw[color=white, use as bounding box] (-11,-0.7) rectangle (4,2.7);
			\draw (0,2) node [label=left:$(2) x$] {} -- (2,2) node [label=right:$y_1 (2)$] {}
			(0,2) -- (2,0) node [label=right:$y_2 (3)$] {}
			(0,0) node [label=left:$(3) x_1$] {} -- (2,2); 
			\node[draw=none,fill=none]  at (-7,1) {$A = \bordermatrix{&y_1&y_2& \ldots 
					\cr x & 1 & 1 & 0 &\ldots & 0
					\cr x_1&1 & 0 & a_{23}&\ldots & a_{2n}
					\cr & 0 & a_{23} & a_{33}&\ldots & a_{3n}
					\cr \vdots & \vdots & \vdots & \vdots & \ddots & \vdots
					\cr & 0 & a_{n2} & a_{n3} &\ldots & a_{nn}}$};
		\end{tikzpicture}   
	\end{center}
	
	By expanding through the row of $x$, $\perm(H) \leq \perm(H-\{x,y_1\})+\perm(H-\{x,y_2\})$. 
	Now assume $H-\{x,y_1\}$ is connected, not isomorphic to $K_2$, $C_6$, or $J$. The $(x,y_{1})-$submatrix has at most $n+k-3$ non-zero entries and, by part (c) of the induction hypothesis, $\perm(H-\{x,y_1\}) \leq c_2 \cdot \alpha^{k-2}$. For the $(x,y_{2})-$submatrix, we can further expand along the column of $y_{1}$ to obtain 
	\begin{equation*}
		\perm(H-\{x,y_2\})=\perm(H-\{x,y_2,y_1,x_1\}) \leq c_2 \cdot \alpha^{k+1-d(x_{1})-d(y_{2})},
	\end{equation*}
	when $H-\{x, x_1, y_1, y_2\}$ is connected, not isomorphic to $K_2$, $C_6$, or $J$, and have minimum degree at most 2. We conclude
	\begin{equation*} 
		\perm(H) \leq c_2 \cdot \alpha^{k-2}+ c_2 \cdot \alpha^{k+1-d(x_{1})-d(y_{2})}\leq c_2 \cdot \alpha^{k-2}+ c_2 \cdot \alpha^{k-5} < 0.945 \cdot c_2 \cdot \alpha^{k},  
	\end{equation*}
	which implies $f(H) < c_2$. Similarly, if $\Delta(H) \leq 3$, then $f(H) < c_1$ by the induction hypothesis of part (b).
\end{proof} 

\subsection{Expanding the permanent along a Type I vertex} \label{sec:type1}

We will see in Claim \ref{claim:Type1Forb} that $H-\{x,y_i\}$ cannot be isomorphic to $K_2$, $C_6$ or $J$, and in Claim \ref{claim:Type1Disc} we use part (a) of the induction hypothesis to obtain that if $H-\{x,y_i\}$ is disconnected then $f(H) \leq c_1$. Recall that $H$ has the following properties: $H$ is bipartite, $C_4$-free, and have maximum degree at least 3.

\begin{Claim} \label{claim:Type1Forb}
	If $x$ is a \emph{Type I} vertex, then $H-\{x,y_1\}$ is not isomorphic to $K_2$, $C_6$ or $J$.
\end{Claim} 

\begin{proof}
	If $H-\{x,y_1\}=K_2$, then $H$ has 4 vertices, which contradicts that $H$ is bipartite and has a vertex of degree at least 3. 
	If $H-\{x,y_1\}=C_6$, since $d(y_1) = 3$, then $y_1$ is adjacent to two vertices of the $C_6$. Since any two vertices in a $C_6$ have distance at most 3, $H$ contains a $C_3$, $C_4$ or $C_5$. It contradicts either the graph being $C_4$-free or bipartite.
	Similarly, since any two vertices of $J$ have distance at most 3, we cannot have $H-\{x,y_1\}=J$.
\end{proof}

\begin{Claim} \label{claim:Type1Disc}
	If $x$ is a \emph{Type I} vertex and $H-\{x,y_1\}$ is disconnected, then $f(H) \leq \alpha^{-1} \leq c_1$.
\end{Claim}

\begin{proof}
	By breaking into cases of when the edges that are going to be deleted are contained in the matching (when computing the permanent) or not, leads us to the following cases.
	\begin{figure}[h!]
		\begin{minipage}{0.50\textwidth}
			\centering
			\caption{Case 1 of Claim \ref{claim:Type1Disc}.}
			\begin{tikzpicture}[thick,scale=1] 
				\draw (0.5,3) node [label=above:$x$] {} -- (2.5,3) node [label=above:$y_1$] {}
				(0.5,3) -- (0,1.5) node [label=below:$y_2$] {}
				(2.5,3) -- (2,1.5)
				(2.5,3) -- (3,1.5);
				\draw (0,0.7) ellipse (0.8 and 1.5);
				\draw (0,0) node [fill=black!0,draw=black!0] {$H_1$};
			\end{tikzpicture}
		\end{minipage}
		\begin{minipage}{0.40\textwidth}
			\centering
			\caption{Case 2 of Claim \ref{claim:Type1Disc}.}
			\begin{tikzpicture}[thick,scale=1]
				\draw (0.5,3) node [label=above:$x$] {} -- (3,3) node [label=above:$y_1$] {}
				(0.5,3) -- (0,1.5) node [label=below:$y_2$] {}
				(3,3) -- (1,1.5) node [label=below:$x_1$] {}
				(3,3) -- (3.5,1.5) node [label=below:$x_2$] {};
				\draw (0.5,0.7) ellipse (1 and 1.5)
				(3,0.7) ellipse (1 and 1.5);
				\draw (0.5,0) node [fill=black!0,draw=black!0] {$H_1$};
				\draw (3,0) node [fill=black!0,draw=black!0] {$H_2$};
			\end{tikzpicture} 
		\end{minipage}
	\end{figure}
	
	\textbf{Case 1:} There is a component containing $y_2$ and not containing any of the neighbors of $y_1$.
	
	If $v(H_1)$ is even, then $xy_2$ cannot be in any perfect matching. Then $H$ and $H-xy_2$ have the same number of perfect matchings, which means $\perm(H) = \perm(H-xy_2)$. Using the inductive hypothesis, we have that $\perm(H-xy_2) \leq \alpha^{e(H-xy_2)-\frac{1}{2}v(H-xy_2)}=\alpha^{(e(H)-1)-\frac{1}{2}v(H)}$. Therefore, $\perm(H) \leq \alpha^{e(H)-\frac{1}{2}v(H)-1}$ and $f(H) \leq \alpha^{-1}$. 
	
	If $v(H_1)$ is odd, then $xy_2$ must be in every perfect matching, then $xy_1$ cannot be in any of them. Then $\perm(H) = \perm(H-xy_1)$. Using the inductive hypothesis, we have that $\perm(H-xy_1) \leq \alpha^{e(H-xy_1)-\frac{1}{2}v(H-xy_1)}=\alpha^{(e(H)-1)-\frac{1}{2}v(H)}$. Therefore, $\perm(H) \leq \alpha^{e(H)-\frac{1}{2}v(H)-1}$ and $f(H) \leq \alpha^{-1}$.
	
	\textbf{Case 2:} The component containing $y_2$ also contains $x_1$, one of the neighbors of $y_1$. 
	
	Since $v(H_1)+v(H_2)$ is even, they can only both be either even or odd.
	
	If $v(H_1)$ and $v(H_2)$ are both even, then $y_1x_2$ cannot be in any perfect matching. Thus $\perm(H) \leq \alpha^{e(H)-\frac{1}{2}v(H)-1}$.
	
	If $v(H_1)$ and $v(H_2)$ are both odd, then $xy_1$ cannot be in any perfect matching. Thus $\perm(H) \leq \alpha^{e(H)-\frac{1}{2}v(H)-1}$.
	
	In conclusion, in all cases $f(H) \leq \alpha^{-1} \leq c_1$.
\end{proof}

\subsection{Expanding the permanent along a Type II vertex} \label{sec:type2}

Now we assume $H$ has no Type I vertex. Let $x$ be a Type II vertex. We will see in Claims \ref{claim:Type2Forb2} and \ref{claim:Type2Forb1} that $H-\{x,y_1\}$ and $H-\{x,x_1,y_1,y_2\}$ cannot be isomorphic to $K_2$, $C_6$ or $J$; and in Claims \ref{claim:Type2Disc1} and \ref{claim:Type2Disc2} we use part (a) of the induction hypothesis to obtain that if $H-\{x,y_1\}$ or $H-\{x,x_1,y_1,y_2\}$ are disconnected then $f(H) \leq c_1$. Otherwise, if $H-\{x,y_1\}$ and $H-\{x,x_1,y_1,y_2\}$ are connected and non-isomorphic to $K_2$, $C_6$, or $J$; then by Claim~\ref{claim:23-23} we conclude 
$$f(H) \leq \alpha^{-2}f(H-\{x,y_1\})+\alpha^{-5}f(H-\{x,x_1,y_1,y_2\}) < c_1, $$ 
when $\Delta(H)\leq 3$. Notice $H-\{x,y_1\}$ has minimum degree at most 2, but we need that $\Delta(H)\leq 3$ to guarantee $H-\{x,x_1,y_1,y_2\}$ has minimum degree at most 2. In general, $H-\{x,x_1,y_1,y_2\}$ has minimum degree at most 2 unless we have the structure of Claim \ref{claim:type2deg4}, in which case we will see that $f(H) < c_2$.

\begin{Claim} \label{claim:Type2Forb2}
	If $x$ is a \emph{Type II} vertex, then $H-\{x,x_1,y_1,y_2\}$ is not isomorphic to $K_2$, $C_6$ or $J$.
\end{Claim} 

\begin{proof}
	If $H-\{x,x_1,y_1,y_2\}=K_2$, since $x_1$ have two neighbors in $H-\{x,x_1,y_1,y_2\}$, we have a copy of $C_3$ in $H$, a contradiction.
	As in the proof of Claim \ref{claim:Type1Forb}, if $H-\{x,x_1,y_1,y_2\}=C_6$, then $x_1$ is adjacent to two vertices of the $C_6$, creating a copy $C_3$, $C_4$ or $C_5$ in $H$.
	Notice that for the graph $J$, any two vertices are at distance at most 3. Similarly, assuming $H-\{x,x_1,y_1,y_2\}=J$, we conclude $H$ has a copy of $C_3$, $C_4$ or $C_5$, a contradicition.
\end{proof}

\begin{Claim} \label{claim:Type2Forb1}
	If $H \neq J$, $H$ has no \emph{Type I} vertex, and $x$ is a \emph{Type II} vertex, then $H-\{x,y_1\}$ is not isomorphic to $K_2$, $C_6$ or $J$.
\end{Claim} 

\begin{proof}
	If $H-\{x,y_1\}=K_2$, then $H$ has 4 vertices, contradicting $x_1, y_2$ have degree 3.
	If $H-\{x,y_1\}=C_6$, then $x_1$ and $y_2$ are vertices of this cycle and either we have a copy $C_3$, $C_4$ or $C_5$ in $H$, or $H$ is isomorphic to $J$.
	
	Notice that we obtain a similar contradiction when $H-\{x,y_1\}=J$, unless $x_1$ and $y_2$ are vertices with distance 3 in $J$. Without loss of generality, we can assume the graph $H$ is the following. 
	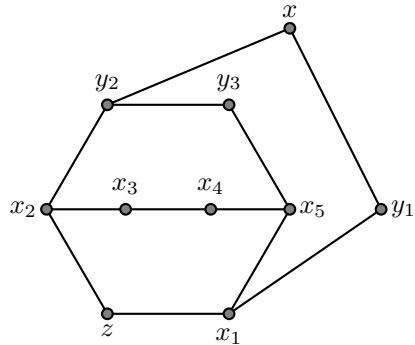
\begin{figure}[h!]
		\centering
		\caption{Only case in Claim \ref{claim:Type2Forb1}}
	\begin{tikzpicture}[thick,scale=0.8]%
		\draw (2,3) -- (3.5,0)
					(-0.7,0) -- (0.7,0);
		\foreach \x in {0,60,120,180,240,300} 
		{\draw (\x:2) -- (\x+60:2);}
		\draw (0:2) -- (0.7,0) node [label=above:$x_4$] {}
		(180:2) -- (-0.7,0) node [label=above:$x_3$] {};
		\draw (120:2) -- (2,3) node [label=above:$x$] {};
		\draw (300:2) -- (3.5,0) node [label=right:$y_1$] {};
		\draw (180:2) node [label=left:$x_2$] {};
		\draw (120:2) node [label=above:$y_2$] {};
		\draw (60:2) node [label=above:$y_3$] {};
		\draw (0:2) node [label=right:$x_5$] {};
		\draw (240:2) node [label=below:$z$] {};
		\draw (300:2) node [label=below:$x_1$] {};
	\end{tikzpicture}
	\end{figure}
	
	\noindent Then $z$ is a \emph{Type I} vertex, a contradiction.
\end{proof}

\begin{Claim} \label{claim:Type2Disc1}
	If $x$ is a \emph{Type II} vertex and $H-\{x,y_1\}$ is disconnected, then $f(H) \leq \alpha^{-1} \leq c_1$.
\end{Claim}

\begin{proof}
	When $H-\{x,y_1\}$ is disconnected, there is only one possible case. Namely, when $y_2$ and $x_1$ belong to different components $H_1$ and $H_2$, respectively, of $H-\{x,y_1\}$.
	\begin{figure}[h!]
		\centering
		\caption{Only case in Claim \ref{claim:Type2Disc1}}
	\begin{tikzpicture}[thick,scale=0.8]%
		\draw (0.5,3) node [label=above:$x$] {} -- (3,3) node [label=above:$y_1$] {}
		(0.5,3) -- (0.5,1.5) node [label=below:$y_2$] {}
		(3,3) -- (3,1.5) node [label=below:$x_1$] {};
		\draw (0.5,0.7) ellipse (1 and 1.5)
		(3,0.7) ellipse (1 and 1.5);
		\draw (0.5,0) node [fill=black!0,draw=black!0] {$H_1$};
		\draw (3,0) node [fill=black!0,draw=black!0] {$H_2$};
	\end{tikzpicture}
	\end{figure}
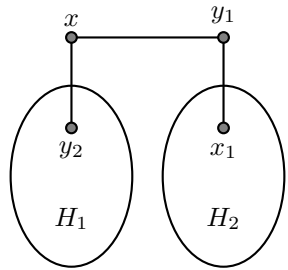
	
	Still, since $v(H_1)+v(H_2)$ is even, they can only both be either even or odd. If $v(H_1)$ and $v(H_2)$ are both even, then $xy_2$ and $y_1x_1$ cannot be in any perfect matching. If $v(H_1)$ and $v(H_2)$ are both odd, then $xy_1$ cannot be an edge in any perfect matching. Thus $\perm(H) \leq \alpha^{e(H)-\frac{1}{2}v(H)-1}$ and in both cases we get $f(H) \leq \alpha^{-1}$.
\end{proof}

\begin{Claim} \label{claim:Type2Disc2}
	If $H$ has no \emph{Type I} vertex, $x$ is a \emph{Type II} vertex, and $H-\{x,x_1,y_1,y_2\}$ is disconnected, then $f(H) \leq c_2$. Moreover, $f(H) \leq c_1$ when $\Delta(H) \leq 3$.
\end{Claim}

\begin{proof}
	Now we assume $H-\{x,x_1,y_1,y_2\}$ is disconnected and break the proof into cases depending on whether the edges that were going to be deleted are contained in the matching or not. We have the following three cases.
	
	\begin{figure}[h!]
		\begin{minipage}{0.33\textwidth}
			\centering
			\caption{Case 1 of Claim \ref{claim:Type2Disc2}.}
			\begin{tikzpicture}[thick,scale=1] 
			\draw (0.5,4) node [label=above:$x$] {} -- (2.5,4) node [label=above:$y_1$] {}
			(0.5,4) -- (0.5,3) node [label=left:$y_2$] {}
			(2.5,4) -- (2.5,3) node [label=right:$x_1$] {}
			(0.5,3) -- (0,1.5) node [label=below:$x_2$] {}
			(0.5,3) -- (1,1.5)
			(2.5,3) -- (2,1.5)
			(2.5,3) -- (3,1.5);
			\draw (0,0.7) ellipse (0.8 and 1.5);
			\draw (0,0) node [fill=black!0,draw=black!0] {$H_1$};
			\end{tikzpicture}
		\end{minipage}
		\begin{minipage}{0.33\textwidth}
			\centering
			\caption{Case 2 of Claim \ref{claim:Type2Disc2}.}
			\begin{tikzpicture}[thick,scale=1]
			\draw (0.5,4) node [label=above:$x$] {} -- (3,4) node [label=above:$y_1$] {}
			(0.5,4) -- (0.5,3) node [label=left:$y_2$] {}
			(3,4) -- (3,3) node [label=right:$x_1$] {}
			(0.5,3) -- (0,1.5) node [label=below:$x_2$] {}
			(0.5,3) -- (1,1.5) node [label=below:$x_3$] {}
			(3,3) -- (2.5,1.5) node [label=below:$y_3$] {}
			(3,3) -- (3.5,1.5) node [label=below:$y_4$] {};
			\draw (0.5,0.7) ellipse (1 and 1.5)
			(3,0.7) ellipse (1 and 1.5);
			\draw (0.5,0) node [fill=black!0,draw=black!0] {$H_1$};
			\draw (3,0) node [fill=black!0,draw=black!0] {$H_2$};
			\end{tikzpicture} 
		\end{minipage}
		\begin{minipage}{0.33\textwidth}
			\centering
			\caption{Case 3 of Claim \ref{claim:Type2Disc2}.}
			\begin{tikzpicture}[thick,scale=1] 
			\draw (0.5,4) node [label=above:$x$] {} -- (3,4) node [label=above:$y_1$] {}
			(0.5,4) -- (0.5,3) node [label=left:$y_2$] {}
			(3,4) -- (3,3) node [label=right:$x_1$] {}
			(0.5,3) -- (0,1.5) node [label=below:$x_2$] {}
			(0.5,3) -- (2.5,1.5) node [label=below:$x_3$] {}
			(3,3) -- (1,1.5) node [label=below:$y_3$] {}
			(3,3) -- (3.5,1.5) node [label=below:$y_4$] {};
			\draw (0.5,0.7) ellipse (1 and 1.5)
			(3,0.7) ellipse (1 and 1.5);
			\draw (0.5,0) node [fill=black!0,draw=black!0] {$H_1$};
			\draw (3,0) node [fill=black!0,draw=black!0] {$H_2$};
			\end{tikzpicture}
		\end{minipage}
	\end{figure}
	
	\textbf{Case 1:}
	There is a component of $H-\{x,x_1,y_1,y_2\}$ containing only one of the four neighbors of $y_2$ and $x_1$. Without loss of generality, we assume $x_2$ is the only such neighbor in the component $H_1$.
	
	If $v(H_1)$ is odd, $x_2y_2$ must be in every perfect matchings, then $xy_2$ cannot be. Thus, $\perm(H) \leq \alpha^{e(H)-\frac{1}{2}v(H)-1}$. 
	
	If $v(H_1)$ is even, then $x_2y_2$ cannot be in any perfect matching, and $\perm(H) \leq \alpha^{e(H)-\frac{1}{2}v(H)-1}$. In both cases we get $f(H) \leq \alpha^{-1}$.
	
	\textbf{Case 2:}
	There are two components in $H-\{x,x_1,y_1,y_2\}$, each containing the two neighbors of $y_2$ or $x_1$. 
	In this case, $H-\{x,y_1\}$ is disconnected and then $f(H)\leq \alpha^{-1}$ by Claim \ref{claim:Type2Disc1}.
	
	\textbf{Case 3:}
	There are two components in $H-\{x,x_1,y_1,y_2\}$, each containing one of the neighbors of both $y_2$ and $x_1$.
	Without loss of generality, we have the following.
	
	If $v(H_1)$ and $v(H_2)$ are both odd, then one of $y_2x_2$, $x_1y_3$ must be in a perfect matching, which means $xy_1$ has to be in the perfect matching and then $xy_2$ and $y_1x_1$ cannot be in any perfect matching. Thus, $\perm(H) \leq \alpha^{e(H)-\frac{1}{2}v(H)-2}$.
	
	If $v(H_1)$ and $v(H_2)$ are both even, we proceed by breaking into cases depending on whether the edges that were going to be deleted are contained in the matching or not.
	
	\sloppy When $xy_1$ is an edge in the perfect matching, then $xy_2$, $y_1x_1$ are not. Then either $y_2x_2$, $x_1y_3$ are edges in a perfect matching, or $y_2x_3$, $x_1y_4$ are edges in a perfect matching. In the first case, there are at most $ \alpha^{ e(H) - \frac{1}{2}v(H) - 4 - (d(x_2)-1) - (d(y_3)-1) + a }$ such perfect matchings; where $a=1$ when $x_2$ is adjacent to $y_3$, and $a=0$, otherwise. In the second case, there are at most $\alpha^{e(H)-\frac{1}{2}v(H)-4-(d(x_3)-1)-(d(y_4)-1)+b}$ such perfect matchings; where $b=1$ when $x_3$ is adjacent to $y_4$, and $b=0$, otherwise.
	
	When $xy_1$ is not an edge in a perfect matching, then $xy_2$ and $y_1x_1$ must be, which implies $y_2x_2$, $y_2x_3$, $x_1y_3$, $x_1y_4$ cannot be in a perfect matching. There are at most $\alpha^{e(H)-\frac{1}{2}v(H)-5}$ such perfect matchings.
	
	Summing up the number of perfect matchings, we get 
	\begin{align*}
		\perm(H) \leq (\alpha^{-(2+d(x_2)+d(y_3)-a)}+\alpha^{-(2+d(x_3)+d(y_4)-b)}+\alpha^{-5}) \cdot \alpha^{e(H)-\frac{1}{2}v(H)}.
	\end{align*}
	If $a=b=0$, we obtain 
	$$f(H) \leq \alpha^{-5}+2\alpha^{-6} \leq c_1 .$$
	If $d(x_2)= 3$ and $a=0$, or when $d(x_2)\geq 4$, we obtain 
	$$f(H) \leq 2\alpha^{-5}+\alpha^{-7} = c_1.$$
	If $d(x_2) = 3$ and $a=1$, then $d(y_3)\geq 3$, as if $d(y_3)=2$, then $y_3$ is a Type I vertex. We get
	$$f(H) \leq 2\alpha^{-5}+\alpha^{-7} = c_1.$$
	
	By symmetry of $x_2$, $y_3$, $x_3$, and $y_4$, the only two cases remained are shown in Figure \ref{fig:case3}, when $d(x_2)=d(x_3)=d(y_3)=d(y_4)=2$. 
	\begin{figure}[h!]
		\centering
		\caption{Case 3: $a=b=1$ and $a=0$, $b=1$.}
	\begin{tikzpicture}[thick,scale=1]%
		\draw (0.5,4) node [label=above:$x$] {} -- (3,4) node [label=above:$y_1$] {}
		(0.5,4) -- (0.5,3) node [label=left:$y_2$] {}
		(3,4) -- (3,3) node [label=right:$x_1$] {}
		(0.5,3) -- (0,1.5) node [label=below:$x_2$] {}
		(0.5,3) -- (2.5,1.5) node [label=below:$x_3$] {}
		(3,3) -- (1,1.5) node [label=below:$y_3$] {}
		(3,3) -- (3.5,1.5) node [label=below:$y_4$] {}
		(2.5,1.5) -- (3.5,1.5)
		(0,1.5) -- (1,1.5);
	\end{tikzpicture} \hspace{3cm} 
	\begin{tikzpicture}[thick,scale=1]%
		\draw (0.5,4) node [label=above:$x$] {} -- (3,4) node [label=above:$y_1$] {}
		(0.5,4) -- (0.5,3) node [label=left:$y_2$] {}
		(3,4) -- (3,3) node [label=right:$x_1$] {}
		(0.5,3) -- (0,1.5) node [label=below:$x_2$] {}
		(0.5,3) -- (2.5,1.5) node [label=below:$x_3$] {}
		(3,3) -- (1,1.5) node [label=below:$y_3$] {}
		(3,3) -- (3.5,1.5) node [label=below:$y_4$] {}
		(2.5,1.5) -- (3.5,1.5);
	\end{tikzpicture} \label{fig:case3}
	\end{figure}
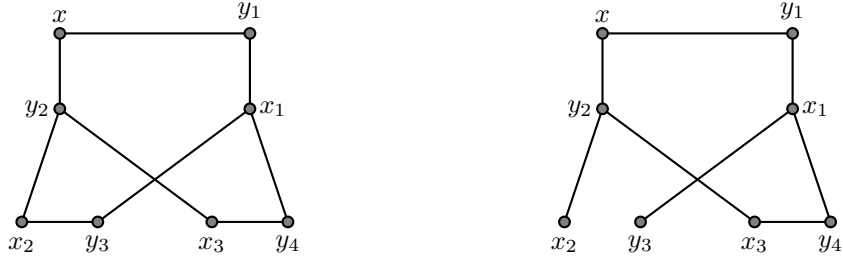
	
	In the first case, $H$ is isomorphic to $J$, a contradiction. In the second case, because $x_2$ is not a Type I vertex, hence it has to be adjacent to a vertex $z_1$ of degree $\neq 3$. If $d(z_1)\geq 4$, then expanding along the vertex $x_2$ yields $f(H)\leq \alpha^{-3}+\alpha^{-4}=c_2$. Otherwise, by symmetry, both $x_2$ and $y_3$ have to be adjacent to vertices with degree $2$, say $z_1$ and $z_2$, which are not adjacent to each other (since, otherwise, there will be a path of vertices with degree 2 of length 3). If the other neighbor of $z_1$ or $z_2$ is a vertex of degree 2, then by Claim \ref{claim:22-23}, we already have the desired bound on $f(H)$. Thus, both $z_1$ and $z_2$ have to be adjacent to vertices with degrees at least 3, say $w_1$ and $w_2$.
	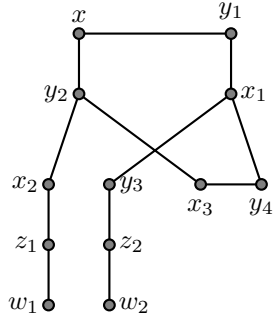
\begin{figure}[h!]
		\centering
		\caption{Case 3: $a=0$, $b=1$.}
		\begin{tikzpicture}[thick,scale=0.8]%
			\draw (0.5,4) node [label=above:$x$] {} -- (3,4) node [label=above:$y_1$] {}
			(0.5,4) -- (0.5,3) node [label=left:$y_2$] {}
			(3,4) -- (3,3) node [label=right:$x_1$] {}
			(0.5,3) -- (0,1.5) node [label=left:$x_2$] {}
			(0.5,3) -- (2.5,1.5) node [label=below:$x_3$] {}
			(3,3) -- (1,1.5) node [label=right:$y_3$] {}
			(3,3) -- (3.5,1.5) node [label=below:$y_4$] {}
			(2.5,1.5) -- (3.5,1.5)
			(0,1.5) -- (0,0.5) node [label=left:$z_1$] {}
			(1,1.5) -- (1,0.5) node [label=right:$z_2$] {}
			(0,0.5) -- (0,-0.5) node [label=left:$w_1$] {}
			(1,0.5) -- (1,-0.5) node [label=right:$w_2$] {};
		\end{tikzpicture}
	\end{figure}
	
	If $d(w_1)\geq 4$, then $$f(H) \leq \alpha^{-2}f(H-\{x_2,z_1\})+\alpha^{-6}f(H-\{x_2,w_1,z_1,y_2\})\leq \alpha^{-2} + \alpha^{-6} \leq c_2.$$
	
	If $d(w_1)=3$, then $x_2$ is a Type II vertex, with neighbors $z_1$ and $y_2$, while $z_1$ is also adjacent to $w_1$. Proceeding as in Claim \ref{claim:23-23}, we get
	$$f(H) \leq \alpha^{-2}f(H-\{x_2,z_1\})+\alpha^{-5}f(H-\{x_2,w_1,z_1,y_2\}).$$
	
	Notice that $H-\{x_2,z_1\}$ have a path of length 4 consisting of degree 2 vertices (namely $y_1xy_2x_3y_4$) and, proceeding as in Claim \ref{claim:22-23}, we have $f(H-\{x_2,z_1\}) \leq \alpha^{-3}+\alpha^{-5}$.
	Further, note that, in $H-\{x_2,w_1,z_1,y_2\}$, $x$ has degree 1 and is adjacent to $y_1$ of degree 2. By expanding along the vertex $x$, we get $f(H-\{x_2,w_1,z_1,y_2\}) \leq \alpha^{-1}$.
	
	We conclude
	\begin{align*}
		f(H) & \leq \alpha^{-2}f(H-\{x_2,z_1\})+\alpha^{-5}f(H-\{x_2,w_1,z_1,y_2\}) \\
		& \leq \alpha^{-2}\cdot (\alpha^{-3}+\alpha^{-5})+\alpha^{-5} \cdot \alpha^{-1} =
		\alpha^{-5}+\alpha^{-6}+\alpha^{-7} \leq c_1. \tag*{\qedhere}
	\end{align*}
\end{proof}

\subsection{Proof of part (c)} \label{sec:proof(c)}

When $x$ is a Type II vertex, $H-\{x,x_1,y_1,y_2\}$ has minimum degree at most 2 unless we have the following structure of Claim \ref{claim:type2deg4}, in which case we will prove that $f(H) < c_2$.

	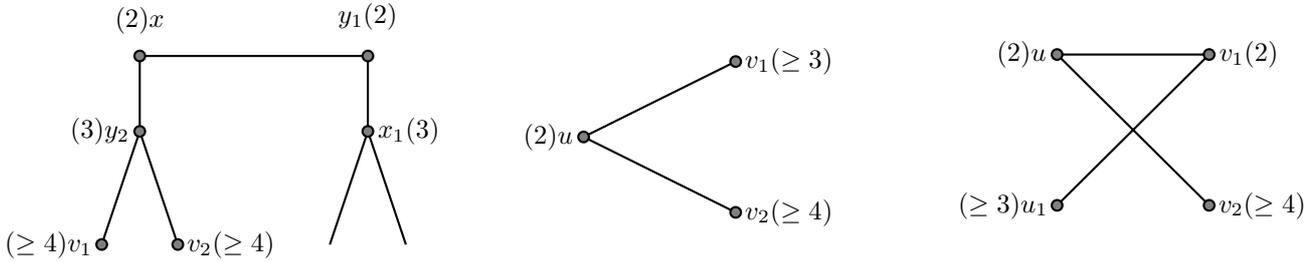
\begin{figure}[h!]
		    \caption{Substructures in Claims \ref{claim:type2deg4}, \ref{claim:2-3,4}, and \ref{claim:23-24}, respectively.}	
	\begin{minipage}{0.33\textwidth}
		\centering
		\begin{tikzpicture}[thick,scale=1] 
		\draw (0.5,4) node [label=above:$(2) x$] {} -- (3.5,4) node [label=above:$y_1 (2)$] {}
		(0.5,4) -- (0.5,3) node [label=left:$(3) y_2$] {}
		(3.5,4) -- (3.5,3) node [label=right:$x_1 (3)$] {}
		(0.5,3) -- (0,1.5) node [label=left:$(\geq 4) v_1$] {}
		(0.5,3) -- (1,1.5) node [label=right:$v_2 (\geq 4)$] {}
		(3.5,3) -- (3,1.5)
		(3.5,3) -- (4,1.5);
		\end{tikzpicture}
	\end{minipage}
	\begin{minipage}{0.33\textwidth}
		\centering
		\begin{tikzpicture}[thick,scale=1]
		\draw (0,1) node [label=left:$(2) u$] {} -- (2,2) node [label=right:$v_1 (\geq 3)$] {}
		(0,1) -- (2,0) node [label=right:$v_2 (\geq 4)$] {};
		\end{tikzpicture} 
	\end{minipage}
	\begin{minipage}{0.33\textwidth}
		\centering
		\begin{tikzpicture}[thick,scale=1] 
		\draw (0,2) node [label=left:$(2) u$] {} -- (2,2) node [label=right:$v_1 (2)$] {}
		(0,2) -- (2,0) node [label=right:$v_2 (\geq 4)$] {}
		(0,0) node [label=left:$(\geq 3) u_1$] {} -- (2,2); 
		\end{tikzpicture}
	\end{minipage} \label{fig:partc}
\end{figure}

\begin{Claim} \label{claim:type2deg4}
	If $x$ is a \emph{Type II} vertex, and $H-\{x,x_1,y_1,y_2\}$ has minimum degree at least 3, then $f(H) \leq \alpha^{-6}+\alpha^{-6}+\alpha^{-5} < c_2$.
\end{Claim}

\begin{proof}
	We use the notation in Figure \ref{fig:partc}. If $H-\{x,x_1,y_1,y_2\}$ has minimum degree at least 3, then before the \emph{Type II} deletion, we can expand on $y_2$ to get  
	$$\perm(H) = \perm(H-\{y_2,v_1\})+\perm(H-\{y_2,v_2\})+\perm(H-\{y_1,x\}).$$ 
	Notice that $d(x)=1$ in $H-\{y_1,v_i\}$, and $d(y_1)=1$ in $H-\{x,y_2\}$. Thus, further expanding along $x$ and $y_1$,
	$$\perm(H) = \perm(H-\{x,y_1,y_2,v_1\}) + \perm(H-\{x,y_1,y_2,v_2\}) + \perm(H-\{x,x_1,y_1,y_2\}).$$
	We conclude 
	$$f(H) \leq \alpha^{-6} + \alpha^{-6} + \alpha^{-5} < 0.8150 < 0.8968 < c_2. \eqno \qedhere$$
\end{proof}

Now, Claims \ref{claim:2-3,4} and \ref{claim:23-24} will use part (a) of the induction hypothesis to deal with all other possible cases when $\Delta(H)>3$. We have the following cases.

\begin{Claim} \label{claim:2-3,4}
	Let $u$ be a vertex with degree 2, and neighbors of $v_1$ and $v_2$. If $d(v_1) \geq 3 $ and $d(v_2) \geq 4$, then 
	$$f(H) \leq \alpha^{-3}+\alpha^{-4} = c_2 .$$
\end{Claim}

\begin{proof}
	Similarly to the proof of Claim \ref{claim:2-3,3}, we have $\perm(H) = \perm(H-\{u,v_1\})+\perm(H-\{u,v_2\}) \leq \alpha^{k-3}+\alpha^{k-4}$. That means $f(H) \leq \alpha^{-3}+\alpha^{-4} = c_2$.
\end{proof}

\begin{Claim} \label{claim:23-24}
	Let $u$ and $v_1$ be adjacent vertices with degree 2. Further, assume that $u$ is adjacent to $v_{2}$, $v_{1}$ is adjacent to $u_{1}$, and $v_{2}$ is not adjacent to $u_{1}$. If $d(u_{1})\geq 3$ and $d(v_{2}) \geq 4$, then $$f(H) \leq \alpha^{-2} + \alpha^{-6} < c_2 .$$
\end{Claim}

\begin{proof}
	Similarly to the proof of Claim \ref{claim:23-23}, $\perm(H) = \perm(H-\{u,v_1\}) + \perm(H-\{u,u_1,v_1,v_2\}) \leq \alpha^{k-2} + \alpha^{k-6}$. That means $f(H) \leq \alpha^{-2}+\alpha^{-6} < 0.88 < 0.8968 < c_2$.
\end{proof}

\section{Proof of part (a)} \label{sec:proof(a)}

As the proof is by induction on $n+k$, we begin with the base case when $n=1$ or $k=0$. When $n=1$, every $1\times 1$ matrix with entries in $\{0,1\}$ has at most $1=1+0$ non-zero entry and permanent at most $1=\alpha^{0}$. When $k=0$, every $n \times n$ matrix with at most $n$ non-zero entries has permanent at most 1, since the determinant is non-zero only when we have exactly one non-zero entry per row and column. 

We first prove, in Sections \ref{sec:deg2}--\ref{sec:mindeg4}, that the result follows unless $G$ is connected, with $\delta(G) = 3$ and $\Delta(G) \leq 5$. We thus deal with the cases $\Delta(G) = 3$, 4, or 5 in Sections \ref{3reg:firststep}, \ref{maxdeg4:firststep}, and \ref{maxdeg5:firststep}, respectively.

\subsection{$G$ is connected and $\delta(G) \leq 2$} \label{sec:deg2}

If $G$ is connected and $\delta(G) \leq 2$, by (c), we have that $f(G)\leq c_2 < 1$ when $G \neq K_2, C_6, J$. As $f(K_2)=f(C_6)=1$ and $f(J)=3\alpha^{-5} < 1$, we conclude $f(G)\leq 1$ for all connected graphs with minimum degree at most 2.

\subsection{$G$ is disconnected} \label{sec:disc}

If $A$, the bi-adjacency matrix of $G$, has a block-diagonal form, say with square matrices $D_1$ and $D_2$ as diagonal blocks, then the result follows by induction, applied to each of the blocks, as $\perm(A)=\perm(D_1)\perm(D_2)$ and the order of the two matrices $D_1$ and $D_2$ add to the order of the matrix $A$. 

\subsection{$\Delta(G)\geq 6$} \label{sec:deg6}

If $A$ contains a line with at least 6 non-zero entries, then we can split $\perm(A)$ as the sum of two permanents with at least 3 ones missing in each. That is, reordering the rows and columns of $A$ if needed, we can assume that
\[
A=\begin{pmatrix} &&&&\ldots & 
\cr &&&&\ldots & 
\cr 1 & 1 & 1 & 1 & 1 & 1& a_{i7}&\ldots&a_{in}
\cr &&&&\ldots & 
\cr &&&&\ldots &
\end{pmatrix},\]
then we have that $\perm(A) =\perm(B)+\perm(C)$, where
\[
B=\begin{pmatrix} &&&&\ldots & 
\cr &&&&\ldots & 
\cr 1 & 1 & 1 & 0 & 0 & 0& a_{i7}&\ldots&a_{in}
\cr &&&&\ldots & 
\cr &&&&\ldots &
\end{pmatrix}
\quad \text{ and } \quad   
C=\begin{pmatrix} &&&&\ldots & 
\cr &&&&\ldots & 
\cr 0 & 0 & 0 & 1 & 1 & 1& 0 &\ldots& 0
\cr &&&&\ldots & 
\cr &&&&\ldots &
\end{pmatrix} .\]

\noindent As both $B$ and $C$ are $n \times n$ matrices with at most $n+k-3$ non-zero entries, the induction hypothesis implies $$\perm(A) \leq \alpha^{k-3}+\alpha^{k-3} = \alpha^{k}.$$

\subsection{$\delta(G)=d \geq 4$} \label{sec:mindeg4}

Let $u$ be a vertex with minimum degree in $G$, and adjacent to $v_1, \dots, v_d$. If $d \geq 4$, then the induction hypothesis after expanding the permanent along the line of $u$ gives 
$$ f(G) \leq \sum_{i=1}^d \alpha^{2-d(u)-d(v_i)} \leq d \cdot \alpha^{2-2d} \leq 1 ,$$
where the last inequality follows from the fact $d \cdot \alpha^{2-2d}$ is decreasing in $d$ for $d \geq 4$, since $\alpha^{2} > 1.25 \geq \frac{d+1}{d}$. For $d=4$, we have $d \cdot \alpha^{2-2d} = 4\alpha^{-6}=1$.

\begin{Remark} \normalfont
	By the discussion above, the result follows unless $G$ is connected, has minimum degree $\delta(G) = 3$ and maximum degree $\Delta(G)\leq5$. We assume this is the case and we deal next with the cases $\Delta(G) = 3, 4, 5$.	
\end{Remark}

\subsection{Connected 3-regular graphs} \label{3reg:firststep}

Assume now $G$ is a 3-regular $C_{4}$-free balanced bipartite graph. Let $u$ be a vertex and $v_1$, $v_2$, $v_3$ its neighbors. Then expanding on the line of $u$ gives
$$f(G) = \alpha^{-4} \cdot \left( \sum_{i=1}^3 f(G-\{u,v_i\}) \right) .$$

Note that $G-\{u,v_1\}$ has minimum degree 2, since $v_3$ has degree 2 in $G-\{u,v_1\}$. Similarly, $G-\{u,v_2\}$ and $G-\{u,v_3\}$ have minimum degree 2 as well. We will see in Claim \ref{claim:1stStepForb} that $G-\{u,v_i\}$ cannot be isomorphic to $K_2$, $C_6$ or $J$. We are left to check what happens when $G-\{u,v_i\}$ is disconnected in Claim \ref{claim:1stStepDisc}. Otherwise, we can use the induction hypothesis of part (b) for $H=G-\{u,v_i\}$ to conclude 
$$f(G) = \alpha^{-4} \cdot \left( \sum_{i=1}^3 f(G-\{u,v_i\}) \right)
\leq \alpha^{-4} \cdot \left( \sum_{i=1}^3 c_1 \right) = 3\alpha^{-4} \cdot c_1 < 1.$$

\begin{Claim} \label{claim:1stStepForb}
	$G-\{u,v_1\}$ is not isomorphic to $K_2$, $C_6$ or $J$.
\end{Claim} 

\begin{proof}
	Note that the graph $G-\{u,v_1\}$ has minimum degree 2 hence $G-\{u,v_1\}$ cannot be a $K_2$. Indeed, the deletion of two adjacent vertices $u$ and $v_1$ decrease the degree of four different vertices (neighbors of $u$ and $v_1$) by 1 and maintain the degree of all other vertices. In particular, $G-\{u,v_1\}$ cannot be $C_6$ or $J$, since both graphs have six vertices of degree 2. 	
\end{proof}

\begin{Claim} \label{claim:1stStepDisc}
	If $G-\{u,v_1\}$ is disconnected, then $f(G) \leq \alpha^{-1} < 1$.
\end{Claim}

\begin{proof}
	If $G-\{u,v_1\}$ is disconnected, then we have 3 cases, depending on how the neighbors of $u$ and $v_1$ are distributed in different components.  
	
	\begin{figure}[h!]
		\begin{minipage}{0.33\textwidth}
			\centering
			\caption{Case 1 in Claim \ref{claim:1stStepDisc}.}
			\begin{tikzpicture}[thick,scale=1] 
			\draw (0.5,3) node [label=above:$u$] {} -- (2.5,3) node [label=above:$v_1$] {}
			(0.5,3) -- (0,1.5) node [label=below:$v_2$] {}
			(0.5,3) -- (1,1.5)
			(2.5,3) -- (2,1.5)
			(2.5,3) -- (3,1.5);
			\draw (0,0.7) ellipse (0.8 and 1.5);
			\draw (0,0) node [fill=black!0,draw=black!0] {$H_1$};
			\end{tikzpicture}
		\end{minipage}
		\begin{minipage}{0.33\textwidth}
			\centering
			\caption{Case 2 in Claim \ref{claim:1stStepDisc}.}
			\begin{tikzpicture}[thick,scale=1]
			\draw (0.5,3) node [label=above:$u$] {} -- (3,3) node [label=above:$v_1$] {}
			(0.5,3) -- (0,1.5) node [label=below:$v_2$] {}
			(0.5,3) -- (1,1.5) node [label=below:$v_3$] {}
			(3,3) -- (2.5,1.5) node [label=below:$u_1$] {}
			(3,3) -- (3.5,1.5) node [label=below:$u_2$] {};
			\draw (0.5,0.7) ellipse (1 and 1.5)
			(3,0.7) ellipse (1 and 1.5);
			\draw (0.5,0) node [fill=black!0,draw=black!0] {$H_1$};
			\draw (3,0) node [fill=black!0,draw=black!0] {$H_2$};
			\end{tikzpicture} 
		\end{minipage}
		\begin{minipage}{0.33\textwidth}
			\centering
			\caption{Case 3 in Claim \ref{claim:1stStepDisc}.}
			\begin{tikzpicture}[thick,scale=1] 
			\draw (0.5,3) node [label=above:$u$] {} -- (3,3) node [label=above:$v_1$] {}
			(0.5,3) -- (0,1.5) node [label=below:$v_2$] {}
			(0.5,3) -- (2.5,1.5) node [label=below:$v_3$] {}
			(3,3) -- (1,1.5) node [label=below:$u_1$] {}
			(3,3) -- (3.5,1.5) node [label=below:$u_2$] {};
			\draw (0.5,0.7) ellipse (1 and 1.5)
			(3,0.7) ellipse (1 and 1.5);
			\draw (0.5,0) node [fill=black!0,draw=black!0] {$H_1$};
			\draw (3,0) node [fill=black!0,draw=black!0] {$H_2$};
			\end{tikzpicture}
		\end{minipage}
	\end{figure}
	
	\textbf{Case 1:} There is a component containing only one of the four neighbors of $u$ and $v_1$. Without loss of generality, we assume $v_2$ is the unique such neighbor in the component $H_1$.
	
	If $v(H_1)$ is odd, $uv_2$ must be in every perfect matching of $G$, hence $uv_3$ cannot be. The number of perfect matchings of $G$ does not change after deleting the edge $uv_3$. By the induction hypothesis, $\perm(G) \leq \alpha^{k-1}$. 
	
	If $v(H_1)$ is even, then $uv_2$ cannot be in any perfect matching, and similarly we have $\perm(G) \leq \alpha^{k-1}$. 
	
	We conclude $f(G) \leq \alpha^{-1} < 1$.
	
	\textbf{Case 2:} There are two components, each containing the two neighbors of $u$ or $v_1$. 
	Let $H_1$ be the component containing $v_2$ and $v_3$, and $H_2$ containing the neighbors of $v_1$ (namely, $u_1$ and $u_2$).
	
	Since $v(H_1)+v(H_2)$ is even, they can only both be either even or odd. If $v(H_1)$ and $v(H_2)$ are both even, then the edges $uv_2$, $uv_3$, $v_1u_1$ and $v_1u_2$ cannot be in any perfect matching of $G$, which implies $\perm(G) \leq \alpha^{k-4}$. If $v(H_1)$ and $v(H_2)$ are both odd, then $uv_1$ cannot be in any perfect matching, implying $\perm(G) \leq \alpha^{k-1}$. 
	
	We conclude $f(G) \leq \alpha^{-1} < 1$.
	
	\textbf{Case 3:} There are two components, each containing one of the neighbors of both $u$ and $v_1$.
	Without loss of generality, we have the following:
	
	If $v(H_1)$ and $v(H_2)$ are both odd, then $uv_1$ cannot be in any perfect matching, implying $\perm(G) \leq \alpha^{k-1}$.
	
	If $v(H_1)$ and $v(H_2)$ are both even, then we break the proof into case according to which edge incident to $u$ is contained in a perfect matching. 
	
	When $uv_1$ is an edge in a perfect matching, then $uv_2$, $v_1u_1$, $v_1u_1$ and $v_1u_2$ are not in a perfect matching. We have at most $\alpha^{k-4}$ such perfect matchings. 
	
	When $uv_2$ is an edge in perfect matching, then $v_1u_1$ must be in the perfect matching and then $uv_1$, $uv_3$, $v_1u_2$ and further edges incident to $v_2$ and $u_1$ cannot be in the perfect matchings. There are at most $\alpha^{k-6}$ such perfect matchings. Similarly, there are at most $\alpha^{k-6}$ perfect matchings containing $uv_3$.
	
	We conclude $\perm(G) \leq \alpha^{k-4}+2\alpha^{k-6}$ and $f(G)\leq \alpha^{-4}+2\alpha^{-6}< 0.8969 < 1$.
\end{proof}

\subsection{Connected graphs with $\delta(G) = 3$ and $\Delta(G)=4$} \label{maxdeg4:firststep}

Let $u$ be a vertex of degree 4 and $v_1$, $v_2$, $v_3$, $v_4$ its neighbors. If $d(v_i) = 4$ for all $i\in \{1,2,3,4\}$, then expanding on the line of $u$ gives $ f(G) \leq  4 \cdot \alpha^{-6} = 1$.
Thus, we can assume that at least one of the neighbors of $u$ is of degree 3. Without loss of generality, assume $d(v_1)=3$. Let $u_1$ and $u_2$ be neighbors of $v_1$ as below.

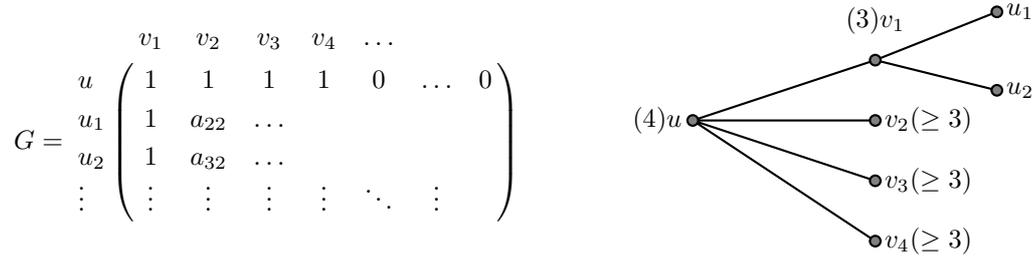
\begin{center}
	\begin{tikzpicture}[thick,scale=0.8]
		\draw[color=white, use as bounding box] (-11,-0.7) rectangle (5,4);
		\draw (0,2) node [label=left:$(4) u$] {} -- (3,3) node [label=above:$(3) v_1 $] {}
		(0,2) -- (3,2)node [label=right:$v_2 (\geq 3)$] {}
		(0,2) -- (3,1)node [label=right:$v_3(\geq 3)$] {}
		(0,2) -- (3,0)node [label=right:$v_4(\geq 3)$] {}
		(3,3) -- (5,3.8)node [label=right:$u_1$]{}
		(3,3) -- (5,2.5)node [label=right:$u_2$]{};
		\node[draw=none,fill=none]  at (-7,2) {$G = \bordermatrix{&v_1&v_2&v_3&v_4&  \ldots
				\cr u & 1 & 1 & 1 & 1& 0& \ldots & 0
				\cr u_1&1 & a_{22} &\ldots
				\cr u_2&1 & a_{32} &\ldots
				\cr \vdots & \vdots & \vdots & \vdots &\vdots & \ddots & \vdots
			} $};
	\end{tikzpicture}   
\end{center}

If $d(u_1) = d(u_2) = 4$, then expansion along the line of $v_1$ would give 
$ f(G) \leq 3 \cdot \alpha^{-5} < 1$. 
Thus, we assume $d(u_1)=3$ and use linearity of the permanent to get $ \perm(G) = \perm(G-uv_1) + \perm(G-\{u,v_1\})$, and then such convert to $f$ as
$$f(G) = \alpha^{-1} \cdot f(G-uv_1) + \alpha^{-5} \cdot f(G-\{u,v_1\}) .$$

We note that both $G-\{u,v_1\}$ and $G-uv_1$ have minimum degree 2, since $u_1$ has degree 2 in $G-\{u,v_1\}$ and $v_1$ has degree 2 in $G-uv_1$. We will see in Claims \ref{claim:missedgenoiso} and \ref{claim:misspointsnoiso} that $G-uv_1$ and $G-\{u,v_1\}$ cannot be isomorphic to $K_2$, $C_6$ or $J$. We are left to check what happens when $G-uv_1$ is disconnected in Claim \ref{claim:g1disconnected}, and when $G-\{u,v_1\}$ is disconnected in Claim \ref{claim:g2disconnected}. Otherwise, we can use the induction hypothesis of part (c) for $H=G-uv_1$ and $H=G-\{u,v_1\}$ to obtain
$$f(G) = \alpha^{-1} \cdot f(G-uv_1) + \alpha^{-5} \cdot f(G-\{u,v_1\}) 
\leq c_2 \cdot ( \alpha^{-1} + \alpha^{-5} ) < 0.9944 < 1 .$$
 
\begin{Claim} \label{claim:missedgenoiso}
	$G-uv_1$ is not isomorphic to $K_2$, $C_6$, or $J$.
\end{Claim}

\begin{proof}
	Note that the graph $G-uv_1$ has minimum degree 2, hence $G-uv_1$ cannot be $K_2$. The deletion of the edge $uv_1$ only decreases the degree of $u$ and $v_1$ by 1 and maintains the degree of all other vertices. In particular, $G-uv_1$ cannot be $C_6$ or $J$, since it would have six vertices of degree 2. 	
\end{proof}

\begin{Claim} \label{claim:misspointsnoiso}
	$G-\{u,v_i\}$ is not isomorphic to $K_2$, $C_6$, or $J$.
\end{Claim} 

\begin{proof}
	Note that the graphs $G-\{u,v_1\}$ has minimum degree 2, hence $G-\{u,v_1\}$ cannot be $K_2$. Indeed, the deletion of two adjacent vertices $u$ and $v_1$ decreases the degree of five different vertices (neighbors of $u$ and $v_1$) by 1 and maintains the degree of all other vertices. In particular, $G-\{u,v_1\}$ cannot be $C_6$ or $J$, since it would have six vertices of degree 2. 	
\end{proof}

\begin{Claim} \label{claim:g1disconnected}
	If $G-uv_1$ is disconnected, then $f(G) \leq \alpha^{-1} < 1$.
\end{Claim}

\begin{proof}
	Assume $G-uv_1$ is disconnected. Let $H_1$ be the component containing $v_1$. Then $u$, $v_2$, $v_3$, $v_4$ are not in $H_1$. 
	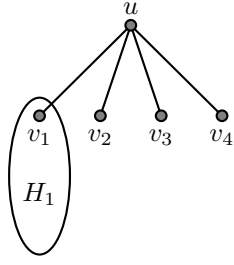
\begin{figure}[h!]
		\centering
		\caption{Only case in Claim \ref{claim:g1disconnected}.}
		\begin{tikzpicture}[thick,scale=0.8]%
			\draw (1.5,3) node [label=above:$u$] {}
			(1.5,3) -- (0,1.5) node [label=below:$v_1$] {}
			(1.5,3) -- (1,1.5) node [label=below:$v_2$] {}
			(1.5,3) -- (2,1.5) node [label=below:$v_3$] {}
			(1.5,3) -- (3,1.5) node [label=below:$v_4$] {};
			\draw (0,0.5) ellipse (0.5 and 1.3);
			\draw (0,0.2) node [fill=black!0,draw=black!0] {$H_1$};
		\end{tikzpicture}
	\end{figure}
	
	If $v(H_1)$ is odd, then $uv_1$ must be in every perfect matching, implying that $uv_2$, $uv_3$, $uv_4$ are not. Thus, $\perm(G) \leq \alpha^{e(G)-\frac{1}{2}v(G)-3}$.
	If $v(H_1)$ is even, then $uv_1$ is not in any perfect matching, implying that $\perm(G) \leq \alpha^{e(G)-\frac{1}{2}v(G)-1}$.
	We conclude $f(G) \leq \alpha^{-1}<1$.
\end{proof}

\begin{Claim} \label{claim:g2disconnected}
	If $G-\{u,v_1\}$ is disconnected, then $f(G) < 1$.
\end{Claim}

\begin{proof}
	If $G-\{u,v_1\}$ is disconnected, then we have 4 cases.
	
		\begin{figure}[h!]
		\begin{minipage}{0.24\textwidth}
			\centering
			\caption{Case 1.}
			\begin{tikzpicture}[thick,scale=0.8] 
				\draw (0.5,3) node [label=above:$u$] {} -- (3,3) node [label=above:$v_1$] {}
				(0.5,3) -- (0,1.5) node [label=below:$v_2$] {}
				(0.5,3) -- (0.5,1.5) node [label=below:$v_3$] {}
				(0.5,3) -- (1,1.5) node [label=below:$v_4$] {}
				(3,3) -- (2.5,1.5) node [label=below:$u_1$] {}
				(3,3) -- (3.5,1.5) node [label=below:$u_2$] {};
				\draw (2.5,0.7) ellipse (0.5 and 1.5);
				\draw (2.5,0.2) node [fill=black!0,draw=black!0] {$H_1$};
			\end{tikzpicture}
		\end{minipage}
		\begin{minipage}{0.24\textwidth}
			\centering
			\caption{Case 2.}
			\begin{tikzpicture}[thick,scale=0.8]
				\draw (0.5,3) node [label=above:$u$] {} -- (3,3) node [label=above:$v_1$] {}
				(0.5,3) -- (0,1.5) node [label=below:$v_2$] {}
				(0.5,3) -- (0.5,1.5) node [label=below:$v_3$] {}
				(0.5,3) -- (1,1.5) node [label=below:$v_4$] {}
				(3,3) -- (2.5,1.5) node [label=below:$u_1$] {}
				(3,3) -- (3.5,1.5) node [label=below:$u_2$] {};
				\draw (0.5,0.7) ellipse (1 and 1.5)
				(3,0.7) ellipse (1 and 1.5);
				\draw (0.5,0) node [fill=black!0,draw=black!0] {$H_1$};
				\draw (3,0) node [fill=black!0,draw=black!0] {$H_2$};
			\end{tikzpicture} 
		\end{minipage}
		\begin{minipage}{0.24\textwidth}
			\centering
			\caption{Case 3.}
			\begin{tikzpicture}[thick,scale=0.8] 
				\draw (0.5,3) node [label=above:$u$] {} -- (3,3) node [label=above:$v_1$] {}
				(0.5,3) -- (0,1.5) node [label=below:$v_2$] {}
				(0.5,3) -- (1,1.5) node [label=below:$v_3$] {}
				(3,3) -- (3,1.5) node [label=below:$u_1$] {}
				(0.5,3) -- (2.5,1.5) node [label=below:$v_4$] {}
				(3,3) -- (3.5,1.5) node [label=below:$u_2$] {};
				\draw (0.5,0.7) ellipse (1 and 1.5)
				(3,0.7) ellipse (1 and 1.5);
				\draw (0.5,0) node [fill=black!0,draw=black!0] {$H_1$};
				\draw (3,0) node [fill=black!0,draw=black!0] {$H_2$};
			\end{tikzpicture}
		\end{minipage}
		\begin{minipage}{0.24\textwidth}
			\centering
			\caption{Case 4.}
			\begin{tikzpicture}[thick,scale=0.8] 
				\draw (0.5,3) node [label=above:$u$] {} -- (3,3) node [label=above:$v_1$] {}
				(0.5,3) -- (0,1.5) node [label=below:$v_2$] {}
				(0.5,3) -- (0.5,1.5) node [label=below:$v_3$] {}
				(3,3) -- (1,1.5) node [label=below:$u_1$] {}
				(0.5,3) -- (2.5,1.5) node [label=below:$v_4$] {}
				(3,3) -- (3.5,1.5) node [label=below:$u_2$] {};
				\draw (0.5,0.7) ellipse (1 and 1.5)
				(3,0.7) ellipse (1 and 1.5);
				\draw (0.5,0) node [fill=black!0,draw=black!0] {$H_1$};
				\draw (3,0) node [fill=black!0,draw=black!0] {$H_2$};
			\end{tikzpicture}
		\end{minipage} 
	\centering \vspace{0.3cm}
	
	Cases in Claim \ref{claim:g2disconnected}.
	\end{figure}
	
	\textbf{Case 1:} There is a component containing only one of the five neighbors of $u$ and $v_1$. We assume $u_1$ is the only such neighbor in the component $H_1$. The other cases can be dealt with similarly.
	
	If $v(H_1)$ is odd, then $v_1u_1$ must be in every perfect matching of G, then $uv_1$ cannot be. The number of perfect matchings of G is the same after deleting the edge $uv_1$. By the induction hypothesis, $\perm(G) \leq \alpha^{e(G)-\frac{1}{2}v(G)-1}$.
	If $v(H_1)$ is even, $v_1u_1$ cannot be in any perfect matching, and similarly we have $\perm(G) \leq \alpha^{e(G)-\frac{1}{2}v(G)-1}$.
	We conclude $f(G) \leq \alpha^{-1}$. Similarly, we have $f(G)\leq \alpha^{-1}$ when $v_2$, $v_3$, $v_4$, or $u_2$ is the unique neighbor in a component of $G-\{u,v_1\}$.
	
	\textbf{Case 2:} There are two components, each containing the neighborhood of $u$ or $v_1$. Let $H_1$ be the
	component containing $v_2$, $v_3$ and $v_4$, and $H_2$ containing $u_1$ and $u_2$.
	
	If $v(H_1)$ and $v(H_2)$ are odd, then $uv_1$ cannot be in any perfect matching of $G$, then $f(G) \leq \alpha^{-1}$.
	If $v(H_1)$ and $v(H_2)$ are even, then $uv_2$, $uv_3$, $uv_4$, $v_1u_1$, and $v_1u_2$ cannot be in any perfect matching, hence $f(G) \leq \alpha^{-5}$.
	
	\textbf{Case 3:} There are two components, one contains two neighbors of $u$, the other one contains two neighbors of $v_1$ and one neighbor of $u$. Assume $H_1$ is a component containing $v_2$, and $v_3$, and $H_2$ containing $v_4$, $u_1$ and $u_2$. 
	
	If $v(H_1)$ and $v(H_2)$ are odd, then $uv_1$ cannot be in any perfect matching, hence $f(G) \leq \alpha^{-1}$.
	If $v(H_1)$ and $v(H_2)$ are even, then $uv_2$ and $uv_3$ cannot be in any perfect matching, hence $f(G) \leq \alpha^{-2}$.
	
	\textbf{Case 4:} There are two components, one contains two neighbors of $u$ and one neighbor of $v_1$, the other component contains one neighbors of $u$ and one neighbor of $v_1$. Assume $H_1$ is a component containing $v_2$, $v_3$, and $u_1$; and $H_2$, containing $v_4$ and $u_2$. 
	
	If $v(H_1)$ and $v(H_2)$ are odd, then $uv_1$ cannot be in any perfect matching, hence $f(G) \leq \alpha^{-1}$.
	
	If $v(H_1)$ and $v(H_2)$ are even, we break into cases depending on which edge incident to $u$ is contained in the matching. 
	The number of perfect matchings containing $uv_1$ is at most $\alpha^{e(G)-\frac{1}{2}v(G)-5}$. 
	When $uv_2$ or $uv_3$ is in a matching, then $v_1u_1$ also must be in the matching and there are at most $2\cdot \alpha^{e(G)-\frac{1}{2}v(G)-7}$ such perfect matchings.
	When $uv_4$ is in a perfect matching, then $v_1u_2$ is also in and there are at most $\alpha^{e(G)-\frac{1}{2}v(G)-7}$ such perfect matchings. 
	Summing up the bounds on the number of perfect matchings, we get $f(G) \leq \alpha^{-5}+3\cdot \alpha^{-7}<0.9103<1$.	
\end{proof}

\subsection{Connected graphs with $\delta(G) = 3$ and $\Delta(G)=5$} \label{maxdeg5:firststep}

Let $u$ be a vertex of degree 5 and $v_1$, $v_2$, $v_3$, $v_4$, $v_5$ its neighbors. If $d(v_i) \geq 4$ for all $i\in \{1,2,3,4,5\}$, then expanding on the line of $u$ gives $ f(G) \leq  5 \cdot \alpha^{-7} < 1$.
Thus, we can assume that at least one of the neighbors is of degree 3, without loss of generality, assume $d(v_1)=3$. Let $u_1$ and $u_2$ be neighbors of $v_1$ as below.

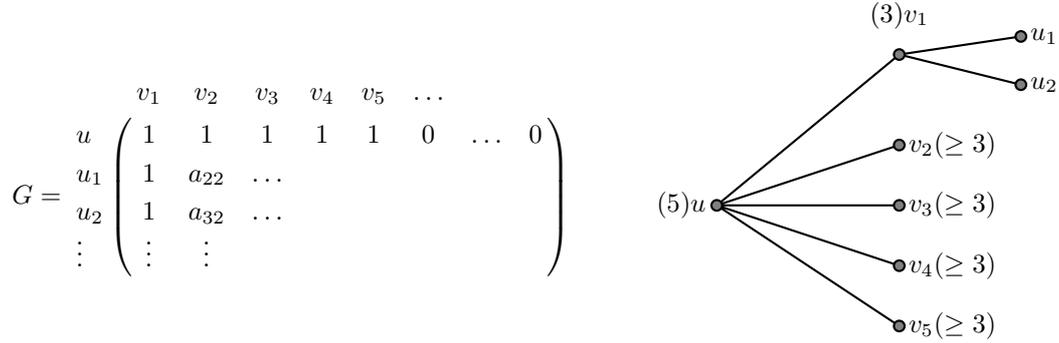
\begin{center}
	\begin{tikzpicture}[thick,scale=0.8]
		\draw[color=white, use as bounding box] (-11,-0.7) rectangle (5,5.5);
		\draw (0,2) node [label=left:$(5) u$] {} -- (3,4.5) node [label=above:$(3) v_1 $] {}
		(0,2) -- (3,3)node [label=right:$v_2 (\geq 3)$] {}
		(0,2) -- (3,2)node [label=right:$v_3 (\geq 3)$] {}
		(0,2) -- (3,1)node [label=right:$v_4(\geq 3)$] {}
		(0,2) -- (3,0)node [label=right:$v_5(\geq 3)$] {}
		(3,4.5) -- (5,4.8)node [label=right:$u_1$]{}
		(3,4.5) -- (5,4)node [label=right:$u_2$]{};
		\node[draw=none,fill=none]  at (-7,2.5) {$G = \bordermatrix{&v_1&v_2&v_3&v_4&v_5&  \ldots
				\cr u & 1 & 1 & 1 & 1& 1 & 0&\ldots & 0
				\cr u_1&1 & a_{22} &\ldots
				\cr u_2&1 & a_{32} &\ldots
				\cr \vdots & \vdots& \vdots
			}  $};
	\end{tikzpicture}   
\end{center}

If $d(u_1),d(u_2) \geq 4$, then the expansion along the line of $v_1$ would give 
$ f(G) \leq 2\alpha^{-5}+\alpha^{-6} < 1$. 
Thus, we assume $d(u_1)=3$ and using the linearity of permanent we get $\perm(G) = \perm(G-uv_1) + \perm(G-\{u,v_1\})$, hence
$$f(G) = \alpha^{-1} \cdot f(G-uv_1) + \alpha^{-6} \cdot f(G-\{u,v_1\}) .$$

We note that both $G-\{u,v_1\}$ and $G-uv_1$ have minimum degree 2. We will see in Claims \ref{claim:d5missedgenoiso} and \ref{claim:d5misspointsnoisod5} that $G-uv_1$ and $G-\{u,v_1\}$ cannot be isomorphic to $K_2$, $C_6$ or $J$. We are left to check what happens when $G-uv_1$ is disconnected in Claim \ref{claim:d5g1disconnected} and $G-\{u,v_1\}$ is disconnected in Claim \ref{claim:d5g2disconnected}. Otherwise, we can use the induction hypothesis of part (c) for $H=G-uv_1$ and $H=G-\{u,v_1\}$ to obtain
$$f(G) = \alpha^{-1} \cdot f(G-uv_1) + \alpha^{-6} \cdot f(G-\{u,v_1\}) 
\leq c_2 \cdot ( \alpha^{-1} + \alpha^{-6} )  < 0.9361 < 1 .$$

\begin{Claim} \label{claim:d5missedgenoiso}
	$G-uv_1$ has minimum degree 2, and it is not isomorphic to $K_2$, $C_6$, or $J$.
\end{Claim}

\begin{proof}
	Note that the graph $G-uv_1$ has minimum degree 2, hence $G-uv_1$ cannot be $K_2$. The deletion of the edge $uv_1$ only decreases the degree of $u$ and $v_1$ by 1 and maintains the degree of all other vertices. In particular, $G-uv_1$ cannot be $C_6$ or $J$, since it would have six vertices of degree 2. 	
\end{proof}

\begin{Claim} \label{claim:d5misspointsnoisod5}
	$G-\{u,v_i\}$ has minimum degree 2, and it is not isomorphic to $K_2$, $C_6$, or $J$.
\end{Claim} 

\begin{proof}
	Note that the graph $G-\{u,v_1\}$ has minimum degree 2, hence $G-\{u,v_1\}$ cannot be $K_2$. 
	Since $C_6$ and $J$ have six vertices of degree 2, if $G-\{u,v_1\}=C_6$ or $G-\{u,v_1\}=J$, then the neighbors of $u$ and $v_1$ must be exactly the six vertices of degree 2, otherwise the minimum degree of $G$ is less than 3. Thus, at least two of $v_2$, $v_3$, $v_4$, $v_5$ are adjacent, which contradicts to the graph being bipartite. 
\end{proof}

\begin{Claim} \label{claim:d5g1disconnected}
	If $G-uv_1$ is disconnected, then $f(G) \leq \alpha^{-1} < 1$.
\end{Claim}

\begin{proof}
	Assume that $G-uv_1$ is disconnected, let $H_1$ be the component containing $v_1$.  Then $u$, $v_2$, $v_3$, $v_4$, and $v_5$ are not in $H_1$.
	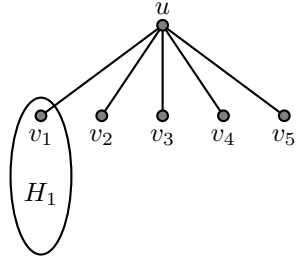
\begin{figure}[h!]
		\centering
		\caption{Only case in Claim \ref{claim:d5g1disconnected}.}
		\begin{tikzpicture}[thick,scale=0.8]%
			\draw (2,3) node [label=above:$u$] {}
			(2,3) -- (0,1.5) node [label=below:$v_1$] {}
			(2,3) -- (1,1.5) node [label=below:$v_2$] {}
			(2,3) -- (2,1.5) node [label=below:$v_3$] {}
			(2,3) -- (3,1.5) node [label=below:$v_4$] {}
			(2,3) -- (4,1.5) node [label=below:$v_5$] {};
			\draw (0,0.5) ellipse (0.5 and 1.3);
			\draw (0,0.2) node [fill=black!0,draw=black!0] {$H_1$};
		\end{tikzpicture}
	\end{figure}
	If $v(H_1)$ is odd, then $uv_1$ must be in every perfect matching, implying that $uv_2$, $uv_3$, $uv_4$, $uv_5$ are not part of any. Thus, $\perm(G) \leq \alpha^{e(G)-\frac{1}{2}v(G)-4}$.
	If $v(H_1)$ is even, then $uv_1$ is not in any perfect matching, implying that $\perm(G) \leq \alpha^{e(G)-\frac{1}{2}v(G)-1}$.
	We conclude $f(G) \leq \alpha^{-1}<1$.
\end{proof}

\begin{Claim} \label{claim:d5g2disconnected}
	If $G-\{u,v_1\}$ is disconnected, then $f(G) < 1$.
\end{Claim}

\begin{proof}
	If $G-\{u,v_1\}$ is disconnected, then we have three cases.
	
		\begin{figure}[h!]
		\begin{minipage}{0.33\textwidth}
			\centering
			\caption{Case 1 in Claim \ref{claim:d5g2disconnected}.}
			\begin{tikzpicture}[thick,scale=1] 
				\draw (0.9,3) node [label=above:$u$] {} -- (4,3) node [label=above:$v_1$] {}
				(0.9,3) -- (0,1.5) node [label=below:$v_2$] {}
				(0.9,3) -- (0.7,1.5) node [label=below:$v_3$] {}
				(0.9,3) -- (1.4,1.5) node [label=below:$v_4$] {}
				(0.9,3) -- (2.1,1.5) node [label=below:$v_5$] {}
				(4,3) -- (3.5,1.5) node [label=below:$u_1$] {}
				(4,3) -- (4.5,1.5) node [label=below:$u_2$] {};
				\draw
				(0.4,0.6) ellipse (0.8 and 1.7);
				\draw (0.4,-0.5) node [fill=black!0,draw=black!0] {$H_1$};
			\end{tikzpicture}
		\end{minipage}
		\begin{minipage}{0.33\textwidth}
			\centering
			\caption{Case 2 in Claim \ref{claim:d5g2disconnected}.}
			\begin{tikzpicture}[thick,scale=1]
				\draw (0.5,3) node [label=above:$u$] {} -- (3,3) node [label=above:$v_1$] {}
				(0.5,3) -- (0,1.5) node [label=below:$v_2$] {}
				(0.5,3) -- (0.5,1.5) node [label=below:$v_3$] {}
				(3,3) -- (1,1.5) node [label=below:$u_1$] {}
				(0.5,3) -- (2.5,1.5) node [label=below:$v_4$] {}
				(0.5,3) -- (3,1.5) node [label=below:$v_5$] {}
				(3,3) -- (3.5,1.5) node [label=below:$u_2$] {};
				\draw (0.5,0.7) ellipse (1 and 1.5)
				(3,0.7) ellipse (1 and 1.5);
				\draw (0.5,0) node [fill=black!0,draw=black!0] {$H_1$};
				\draw (3,0) node [fill=black!0,draw=black!0] {$H_2$};
			\end{tikzpicture} 
		\end{minipage}
		\begin{minipage}{0.33\textwidth}
			\centering
			\caption{Case 3 in Claim \ref{claim:d5g2disconnected}.}
			\begin{tikzpicture}[thick,scale=1] 
				\draw (0.5,3) node [label=above:$u$] {} -- (3,3) node [label=above:$v_1$] {}
				(0.5,3) -- (-0.3,1.5) node [label=below:$v_2$] {}
				(0.5,3) -- (0.1,1.5) node [label=below:$v_3$] {}
				(0.5,3) -- (0.5,1.5) node [label=below:$v_4$] {}
				(3,3) -- (3.3,1.5) node [label=below:$u_2$] {}
				(0.5,3) -- (2.8,1.5) node [label=below:$v_5$] {}
				(3,3) -- (1,1.5) node [label=below:$u_1$] {};
				\draw (0.5,0.7) ellipse (1.3 and 1.5)
				(3.2,0.7) ellipse (1 and 1.5);
				\draw (0.5,0) node [fill=black!0,draw=black!0] {$H_1$};
				\draw (3,0) node [fill=black!0,draw=black!0] {$H_2$};
			\end{tikzpicture}
		\end{minipage}
	\end{figure}
	
	\textbf{Case 1:} There is a component containing neighbors of only one of $u$ or $v_1$. 
	
	Assume that $H_1$ is such component, without loss of generality, $v_2\in H_1$, hence by assumption, $H_1$ contains no neighbor of $v_1$. 
	If $v(H_1)$ is odd, then $uv_1$ cannot be in any perfect matching of G, hence $f(G) \leq \alpha^{-1}$.
	If $v(H_1)$ is even, then $uv_2$ cannot be in any perfect matching, hence $f(G) \leq \alpha^{-1}$.
	
	\textbf{Case 2:} There are two components, both contain two neighbors of $u$ and one neighbor of $v_1$. We assume $v_2$, $v_3$, and $u_1$ are in the component $H_1$ and $v_4$, $v_5$, and $u_2$ are in $H_2$.
	
	If $v(H_1)$ and $v(H_2)$ are odd, then $uv_1$ cannot be in any perfect matching, hence $f(G) \leq \alpha^{-1}$.
	
	If $v(H_1)$ and $v(H_2)$ are even, we break into cases of which edge incident to $u$ is contained in the matching. 
	We have at most $\alpha^{e(G)-\frac{1}{2}v(G)-6}$ perfect matchings containing $uv_1$.
	When $uv_2$ or $uv_3$ is in a perfect matching, then $v_1u_1$ also must be in that matching and there are at most $2\cdot \alpha^{e(G)-\frac{1}{2}v(G)-8}$ such perfect matchings.
	Similarly, when $uv_4$ or $uv_5$ is in a perfect matching, then $v_1u_2$ also is and there are at most $2\cdot \alpha^{e(G)-\frac{1}{2}v(G)-8}$ such perfect matchings. 
	Summing up the bounds on the number of perfect matchings, we get $f(G) \leq \alpha^{-6}+4\cdot \alpha^{-8} < 0.88 <1$.	

	\textbf{Case 3:} There are two components, one contains three neighbors of $u$ and one neighbor of $v_1$. We assume $v_2$, $v_3$, $v_4$ and $u_1$ are in the component $H_1$, and $v_5$ and $u_2$ are in $H_2$.
	
	If $v(H_1)$ and $v(H_2)$ are odd, then $uv_1$ cannot be in any perfect matching, hence $f(G) \leq \alpha^{-1}$.
	
	If $v(H_1)$ and $v(H_2)$ are even, we break into cases of which edge incident to $u$ is contained in the matching. 
	We have at most $\alpha^{e(G)-\frac{1}{2}v(G)-6}$ perfect matchings containing $uv_1$. 
	When $uv_2$, $uv_3$, or $uv_4$ is in a perfect matching, then $v_1u_1$ also must be in and there are at most $3\cdot \alpha^{e(G)-\frac{1}{2}v(G)-8}$ such perfect matchings.
	When $uv_5$ is in the matching, $v_1u_2$ also is and there are at most $\alpha^{e(G)-\frac{1}{2}v(G)-8}$ such perfect matchings. 
	Summing up the bounds on the number of perfect matchings, we get $f(G) \leq \alpha^{-6}+4\cdot \alpha^{-8}<0.88<1$.		
\end{proof}

\section{Determinant of graphs containing a $C_4$} \label{sec:detproof}

We notice the cofactor expansion $\perm(G) = \sum\limits_{i=1}^t \perm(G-\{u,v_i\})$ for the determinant read as 
$$\det(G) \leq \sum\limits_{i=1}^t \det(G-\{u,v_i\}).$$ Using the analogous auxiliary function $f'(G) = \alpha^{-e(G)+\frac{1}{2}v(G)} \cdot \det (G)$, we can mimic the proof of Theorem \ref{thm:maxdeg3,4,5} to obtain Theorem \ref{thm:det}. The only places where we used the $C_4$-free assumption in the proof above was in Claims \ref{claim:Type1Forb}, \ref{claim:Type2Forb2}, and \ref{claim:Type2Forb1}.

We now discuss how to deal with these cases assuming the graph $G$ can potentially have $C_4$ as a subgraph, but we shall bound the determinant rather than the permanent. We highlight that the assumption of $C_4$-free for the bound on permanents is needed, since $\perm(C_4) = 2 > \alpha^{2}$.

For the next cases, it will be useful to first consider when there is a vertex $u$ of degree 2 contained in a $C_4$. 

\begin{Claim} \label{claim:c4deg2}
	If $G$ has a vertex $u$ of degree 2 contained in a $C_4$, then $f'(G) \leq \alpha^{-2} < c_1$.  
\end{Claim}

\begin{proof}
If $u$, $v_1$, $u_1$ and $v_2$ induce a $C_4$ in $G$ then   
\[A = \bordermatrix{&v_1&v_2& \ldots 
	\cr u & 1 & 1 & 0 &\ldots & 0
	\cr u_1&1 & 1 & a_{23}&\ldots & a_{2n}
	\cr & a_{13} & a_{23} & a_{33}&\ldots & a_{3n}
	\cr \vdots & \vdots & \vdots & \vdots & \ddots & \vdots
	\cr & a_{n1} & a_{n2} & a_{n3} &\ldots & a_{nn}} ,
\] 
and we can subtract the line of $u$ from the line of $u_1$, which does not change the determinant of the matrix. We obtain, using the induction hypothesis, $$\det(G)=\det(G-\{u_1v_1,u_1v_2\})\leq \alpha^{k-2} < c_1 \cdot \alpha^{k}. \eqno \qedhere$$
\end{proof}

Since we used that the graph is $C_4$-free in Claims \ref{claim:Type1Forb}, \ref{claim:Type2Forb2}, \ref{claim:Type2Forb1}, and \ref{claim:d5misspointsnoisod5}, so we get rid of those cases when there is a $C_4$, by explicitly computing the value of determinant in the following claims. 

\begin{Claim}
	If $x$ is a Type I vertex and $H-\{x,y_1\}=C_6$, then $f'(H) \leq c_1$.
\end{Claim}

\begin{proof}
	If $x$ is a Type I vertex and $H-\{x,y_1\}=C_6$, then $y_2$ is a vertex in the $C_6$, and $y_1$ is adjacent to another two vertices in the $C_6$. Given that $H$ is bipartite, $x$ is contained in a $C_4$. By Claim \ref{claim:c4deg2}, $f'(H) \leq \alpha^{-2} < c_1$.  
\end{proof}

\begin{Claim} \label{claim:det33}
	If $x$ is a Type I vertex and $H-\{x,y_1\}=J$, then $f'(H) \leq c_1$.
\end{Claim}

\begin{proof}
	If $x$ is a Type I vertex and $H-\{x,y_1\}=J$, then $y_2$ is a vertex of degree 2 in $J$, and $y_1$ is adjacent to another two vertices in $J$. By Claim \ref{claim:c4deg2}, we can assume that $x$ is not in a $C_4$. Using the notation on Figure \ref{fig:det33}, observe that $y_1$ can only be adjacent to the $z_i$'s, otherwise $H$ is not bipartite, and $y_1$ cannot be adjacent to $z_1$ or $z_2$, otherwise $x$ is in a $C_4$. Therefore $y_1$ is adjacent to $z_3$ and $z_4$. 
	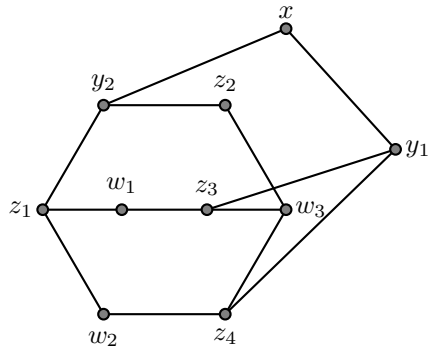
\begin{figure}[h!] 
		\centering 
		\caption{Only case in Claim \ref{claim:det33}}
		\begin{tikzpicture}[thick,scale=0.8]%
			\draw (2,3) -- (3.8,1);
			\draw (3.8,1) -- (0.7,0);
			\draw (300:2) -- (3.8,1) node [label=right:$y_1$] {};
			\draw (120:2) -- (2,3) node [label=above:$x$] {};
			\foreach \x in {0,60,120,180,240,300} 
			{\draw (\x:2) -- (\x+60:2);}
			\draw (0:2) -- (0.7,0) node [label=above:$z_3$] {}
			(180:2) -- (-0.7,0) node [label=above:$w_1$] {}
			(-0.7,0) -- (0.7,0);
			\draw (180:2) node [label=left:$z_1$] {};
			\draw (120:2) node [label=above:$y_2$] {};
			\draw (60:2) node [label=above:$z_2$] {};
			\draw (0:2) node [label=right:$w_3$] {};
			\draw (240:2) node [label=below:$w_2$] {};
			\draw (300:2) node [label=below:$z_4$] {};	
		\end{tikzpicture} \label{fig:det33}
	\end{figure}

	\noindent The determinant of the corresponding bi-adjacency matrix is $5$. 
	Notice that $k=8$, hence $$ f'(H) = 5\alpha^{-8} < 0.7875 < 0.8283 < c_1. \eqno \qedhere $$	
\end{proof}

\begin{Claim} \label{claim:det34}
	If $H$ has no Type I vertex, $x$ is a Type II vertex and $H-\{x,y_1,x_1,y_2\}=C_6$, then $f'(H) \leq c_1$.
\end{Claim}

\begin{proof}
	If $x$ is a Type II vertex and $H-\{x,y_1,x_1,y_2\}=C_6$, then both $x_1$ and $y_2$ are connected to two vertices in the $C_6$. First notice that $x_1$ and $x_2$ cannot have a common neighbor, since this would create a $C_5$. Further, the neighbors of $x_1$ (and of $y_2$) must have distance 2 to avoid $C_3$ or $C_5$. 
	
	Assume $x_1$ is adjacent to $w_1$ and $w_3$ as below. If $y_2$ is not adjacent to $w_2$, which is shown in the first graph below, then $w_2$ is a Type I vertex, a contradiction.
	If $y_2$ is adjacent to $w_2$, we will have the second graph below, up to isomorphism. 
	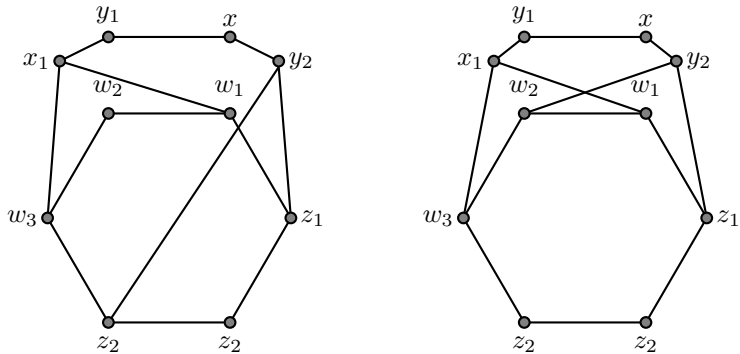
\begin{figure}[h!]
		\centering
		\caption{Cases of Claim \ref{claim:det34}}
	\begin{tikzpicture}[thick,scale=0.8]%
		\draw (-1,3) -- (1,3);
		\draw (-1.8,2.6) -- (60:2);
		\draw (-1.8,2.6) -- (180:2);
		\draw (1.8,2.5) -- (240:2);
		\draw (1.8,2.5) -- (0:2);
		\foreach \x in {0,60,120,180,240,300} 
		{\draw (\x:2) -- (\x+60:2);}
		(-0.7,1) -- (0.7,1);
		\draw (180:2) node 
		[label=left:$w_3$]{};
		\draw (120:2) node 
		[label=above:$w_2$]{};
		\draw (60:2) node 
		[label=above:$w_1$]{};
		\draw (0:2) node 
		[label=right:$z_1$] {};
		\draw (240:2) node [label=below:$z_2$] {};
		\draw (300:2) node  
		[label=below:$z_2$] {};
		\draw (1.8,2.6) node [label=right:$y_2$] {}-- (1,3) node [label=above:$x$] {};
		\draw (-1.8,2.6) node [label=left:$x_1$] {} -- (-1,3) node [label=above:$y_1$] {};
	\end{tikzpicture} \hspace{1cm}
	\begin{tikzpicture}[thick,scale=0.8]%
		\draw (-1,3) -- (1,3);
		\draw (-1.5,2.6) -- (60:2);
		\draw (-1.5,2.6) -- (180:2);
		\draw (1.5,2.6) -- (120:2);
		\draw (1.5,2.6) -- (0:2);
		\foreach \x in {0,60,120,180,240,300} 
		{\draw (\x:2) -- (\x+60:2);}
		(-0.7,1) -- (0.7,1);
		\draw (180:2) node 
		[label=left:$w_3$]{};
		\draw (120:2) node 
		[label=above:$w_2$]{};
		\draw (60:2) node 
		[label=above:$w_1$]{};
		\draw (0:2) node 
		[label=right:$z_1$] {};
		\draw (240:2) node [label=below:$z_2$] {};
		\draw (300:2) node  
		[label=below:$z_2$] {};
		\draw (1.5,2.6) node [label=right:$y_2$] {}-- (1,3) node [label=above:$x$] {};
		\draw (-1.5,2.6) node [label=left:$x_1$] {} -- (-1,3) node [label=above:$y_1$] {};	
	\end{tikzpicture}
	\end{figure}
		
	\noindent The determinant of the corresponding bi-adjacency matrix is 4. Notice that $k=8$, hence $$f'(H) = 4\alpha^{-8} < 0.63 < 0.8283 < c_1. \eqno \qedhere$$
\end{proof}

\begin{Claim} \label{claim:det35}
	If $H$ has no Type I vertex, $x$ is a Type II vertex and $H-\{x,y_1,x_1,y_2\}=J$, then $f'(H) \leq c_1$.
\end{Claim}

\begin{proof}
	If $x$ is a Type II vertex and $H-\{x,y_1,x_1,y_2\}=J$, then both $x_1$ and $y_2$ are connected to two vertices in $J$. Similarly to discussion in the last claim, $x_1$ and $y_2$ cannot have a common neighbor and the neighbors of $x_1$ or $y_2$ must have distance 2. 
		
	We have two cases as in Figure \ref{fig:det35}. 
	In the first case, assume the neighbors of $y_2$ are vertices of degree 2 in $J$. Since $H$ is bipartite, we can assume that the neighbors of $y_2$ are $x_2$ and $x_3$. Therefore $y_3$ and $y_4$ must be neighbors of $x_1$, otherwise, they would be Type I vertex. The determinant of the corresponding bi-adjacency matrix is 5. Notice that $k=10$, hence 
	$$f'(H)=5\alpha^{-10} < 0.4961 < 0.8283 < c_1.$$
	
	In the second case, assume $y_2$ has a neighbor with degree 3 in $J$, say $x_2$. Then the other neighbor of $y_2$ must be one of $x_3$, $x_4$, or $x_5$. Without losss of generality, let us assume it is $x_3$. 
	Thus $y_3$ must be a neighbor of $x_1$, otherwise, it would be a Type I vertex. Consequently, the only choice for the other neighbor of $x_1$ is $y_6$, otherwise, $x_4$ or $x_5$ would be a Type I vertex. 
	The determinant of the corresponding bi-adjacency matrix is 6. Notice that $k=10$, hence $$f'(H)=6\alpha^{-10} < 0.5953 < 0.8283 < c_1. \eqno \qedhere $$
	
	\begin{figure}[h!]
		\centering
		\caption{Cases of Claim \ref{claim:det35}}
	\begin{tikzpicture} [thick,scale=0.8]%
		\draw (-1,3) -- (1,3);
		\draw (2,2.5) -- (0.7,0);
		\draw (2,2.5) -- (60:2);
		\draw (-2,2.5) -- (120:2);
		\draw (-2,2.5) -- (-0.7,0);
		\foreach \x in {0,60,120,180,240,300} 
		{\draw (\x:2) -- (\x+60:2);}
		\draw (0:2) -- (0.7,0) node [label=below:$y_4$] {}
		(180:2) -- (-0.7,0) node [label=below:$x_3$] {}
		(-0.7,0) -- (0.7,0);
		\draw (180:2) node [label=left:$y_6$] {};
		\draw (120:2) node [label=above:$x_2$] {};
		\draw (60:2) node [label=above:$y_3$] {};
		\draw (0:2) node [label=right:$x_5$] {};
		\draw (240:2) node [label=below:$x_4$] {};
		\draw (300:2) node [label=below:$y_5$] {};

		\draw (-2,2.5) node [label=above:$y_2$] {}-- (-1,3) node [label=above:$x$] {};
		
		\draw (2,2.5) node [label=above:$x_1$] {} -- (1,3) node [label=above:$y_1$] {};
	\end{tikzpicture} \hspace{1cm}
	\begin{tikzpicture} [thick,scale=0.8]%
		\draw (-1,3) -- (1,3);
		\draw (2,2.7) -- (0:2);
		\draw (2,2.7) -- (120:2);
		\draw (-2,2.7) -- (60:2);
		\draw (-2,2.7) -- (180:2);
		\foreach \x in {0,60,120,180,240,300} 
		{\draw (\x:2) -- (\x+60:2);}
		\draw (0:2) -- (0.7,0) node [label=below:$x_4$] {}
		(180:2) -- (-0.7,0) node [label=below:$y_4$] {}
		(-0.7,0) -- (0.7,0);
		\draw (180:2) node [label=left:$x_2$] {};
		\draw (120:2) node [label=above:$y_3$] {};
		\draw (60:2) node [label=above:$x_3$] {};
		\draw (0:2) node [label=right:$y_6$] {};
		\draw (240:2) node [label=below:$y_5$] {};
		\draw (300:2) node [label=below:$x_5$] {};

		\draw (-2,2.7) node [label=above:$y_2$] {}-- (-1,3) node [label=above:$x$] {};
		
		\draw (2,2.7) node [label=above:$x_1$] {} -- (1,3) node [label=above:$y_1$] {};
	\end{tikzpicture} \label{fig:det35}
	\end{figure}
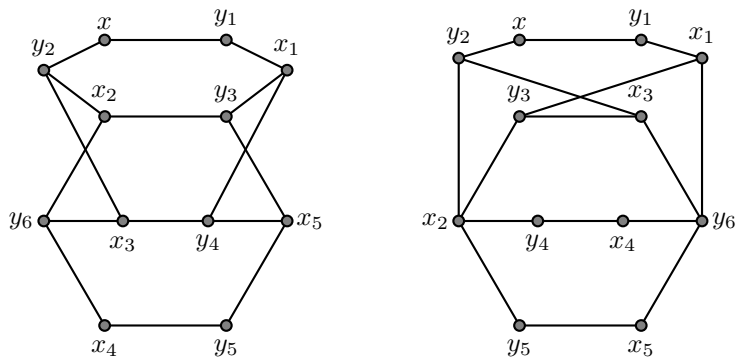
\end{proof}

\begin{Claim} \label{claim:det36}
	If $H\neq J$, $H$ has no Type I vertex, $x$ is a Type II vertex and $H-\{x,y_1\}=C_6$, then $f'(H) \leq c_1$.
\end{Claim}

\begin{proof}
	If $x$ is a Type II vertex and $H-\{x,y_1\}=C_6$, then both $x_1$ and $y_2$ are vertices of the $C_6$. We have 3 cases, up to isomorphism, shown below. 
	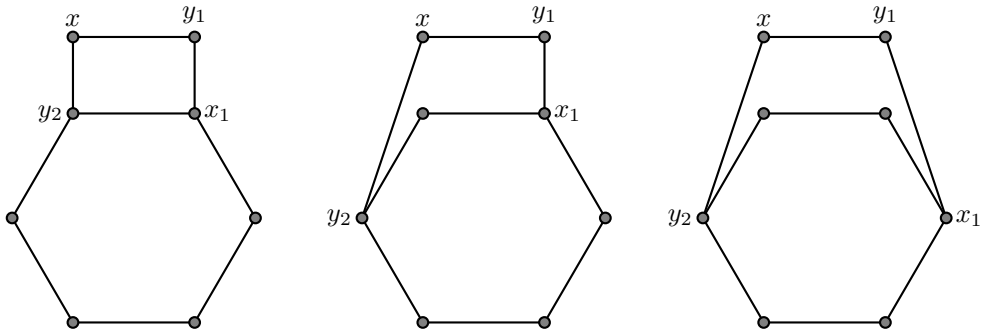
\begin{figure}[h!]
		\centering
		\caption{Cases in Claim \ref{claim:det36}.}
		\begin{tikzpicture}[thick,scale=0.8]%
			\draw (-1,3) -- (1,3);
			\draw (-1,3) -- (120:2);
			\draw (1,3) -- (60:2);
			\foreach \x in {0,60,120,180,240,300} 
			{\draw (\x:2) -- (\x+60:2);}
			(-0.7,1) -- (0.7,1);
			\draw (180:2) node {};
			\draw (120:2) node 
			[label=left:$y_2$] {};
			\draw (60:2) node 
			[label=right:$x_1$] {};
			\draw (0:2) node {};
			\draw (240:2) node {};
			\draw (300:2) node {};	
			\draw  (-1,3) node [label=above:$x$] {};	
			\draw (1,3) node [label=above:$y_1$] {};
		\end{tikzpicture} \qquad
		\begin{tikzpicture}[thick,scale=0.8]%
			\draw (-1,3) -- (1,3);
			\draw (-1,3) -- (180:2);
			\draw (1,3) -- (60:2);
			\foreach \x in {0,60,120,180,240,300} 
			{\draw (\x:2) -- (\x+60:2);}
			(-0.7,1) -- (0.7,1);
			\draw (180:2) node 
			[label=left:$y_2$] {};
			\draw (120:2) node  {};
			\draw (60:2) node 
			[label=right:$x_1$] {};
			\draw (0:2) node {};
			\draw (240:2) node {};
			\draw (300:2) node {};
			
			\draw  (-1,3) node [label=above:$x$] {};
			
			\draw (1,3) node [label=above:$y_1$] {};
		\end{tikzpicture}\qquad
		\begin{tikzpicture}[thick,scale=0.8]%
			\draw (-1,3) -- (1,3);
			\draw (-1,3) -- (180:2);
			\draw (1,3) -- (0:2);
			\foreach \x in {0,60,120,180,240,300} 
			{\draw (\x:2) -- (\x+60:2);}
			(-0.7,1) -- (0.7,1);
			\draw (180:2) node 
			[label=left:$y_2$] {};
			\draw (120:2) node {};
			\draw (60:2) node {};
			\draw (0:2) node 
			[label=right:$x_1$]{};
			\draw (240:2) node {};
			\draw (300:2) node {};
			
			\draw  (-1,3) node [label=above:$x$] {};
			
			\draw (1,3) node [label=above:$y_1$] {};
		\end{tikzpicture} 
	\end{figure}
	
	In the first graph, $d(x_1, y_2)=1$, and $x$ is in a $C_4$, as before we have $f'(H) \leq \alpha^{-2} < c_1$. In the second graph, $x$ is in a $C_5$, a contradiction. For the third graph, $H=J$, a contradiction.
\end{proof}

\begin{Claim} \label{claim:det37}
	If $H$ has no Type I vertex, $x$ is a Type II vertex, and $H-\{x,y_1\}=J$, then $f'(H) \leq c_1$.
\end{Claim}

\begin{proof}
	If $x$ is a Type II vertex and $H-\{x,y_1\}=J$, then both $x_1$ and $y_2$ are vertices of degree 2 in $J$. As $H$ is bipartite, we have two cases, up to isomorphism, shown below.
	
	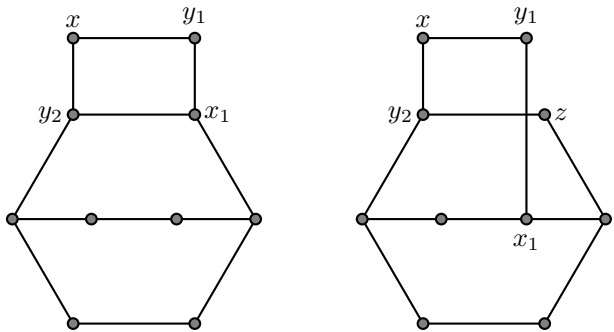
\begin{figure}[h!]
		\centering
		\caption{Cases in Claim \ref{claim:det37}.}
	\begin{tikzpicture}[thick,scale=0.8]%
		\draw (-1,3) -- (1,3);
		\draw (-1,3) -- (120:2);
		\draw (1,3) -- (60:2);
		\foreach \x in {0,60,120,180,240,300} 
		{\draw (\x:2) -- (\x+60:2);}
		\draw (0:2) -- (0.7,0) node {}
		(180:2) -- (-0.7,0) node {}
		(-0.7,0) -- (0.7,0);
		\draw (180:2) node {};
		\draw (120:2) node [label=left:$y_2$] {};
		\draw (60:2) node [label=right:$x_1$] {};
		\draw (0:2) node {};
		\draw (240:2) node {};
		\draw (300:2) node {};
		
		\draw  (-1,3) node 
		[label=above:$x$] {};
		\draw (1,3) node 
		[label=above:$y_1$] {};
	\end{tikzpicture} \hspace{1cm}
	\begin{tikzpicture}[thick,scale=0.8]%
		\draw (-1,3) -- (0.7,3);
		\draw (-1,3) -- (120:2);
		\draw(0.7,3) -- (0.7,0);
		\foreach \x in {0,60,120,180,240,300} 
		{\draw (\x:2) -- (\x+60:2);}
		\draw (0:2) -- (0.7,0) node 
		[label = below:$x_1$] {}
		(180:2) -- (-0.7,0) node {}
		(-0.7,0) -- (0.7,0);
		\draw (180:2) node {};
		\draw (120:2) node [label=left:$y_2$] {};
		\draw (60:2) node 
		[label=right:$z$]{};
		\draw (0:2) node {};
		\draw (240:2) node {};
		\draw (300:2) node {};
		
		\draw  (-1,3) node 
		[label=above:$x$] {};
		\draw (0.7,3) node 
		[label=above:$y_1$] {};
	\end{tikzpicture} \label{fig:det37}
	\end{figure}

	 \noindent Either $x$ in contained a $C_4$, see the graph on the left side on Figure \ref{fig:det37}, or there is a Type I vertex, namely $z$, see the graph on the right side on Figure \ref{fig:det37}. We conclude both cases were handled before, hence $f'(H) \leq c_1$.
\end{proof}

\section{Final remarks and open problems}

We conclude that even that the initial aspiration was to prove Corollary \ref{cor:det}, which is tight for every $n$ multiple of 3, our generalized result of Theorem \ref{thm:det} is best possible only when $k \leq n$. In particular, Corollary \ref{cor:det3n} is not expected to be optimal. As mentioned before, Bruhn and Rautenbach \cite{Bruhn} noted that the incidence matrix of the Fano plane has determinant 24. The graph formed by vertex disjoint copies of them gives a lower bound of $24^{n/7}$ for the maximum determinant of matrices with at most $3n$ ones.

\begin{Conjecture}[Bruhn, Rautenbach]
	If $A \in \{0, 1\}^{n\times n}$ has at most $3n$ non-zero entries, then $\det(A) \leq 24^{n/7}$.
\end{Conjecture}

Similar questions for permanents can also be examined. Somewhat surprisingly, the permanent of the incidence matrix of the Fano plane is equal to its determinant. We conjecture that this is the maximum permanent among $C_4$-free bipartite graphs as well. Note that it would be that a variant of our method might solve this conjecture, we have not attempted to do so.

\begin{Conjecture}
		If $A \in \{0, 1\}^{n\times n}$ is $C_4$-free and has at most $3n$ non-zero entries, then $\perm(A) \leq 24^{n/7}$.
\end{Conjecture}

Intuitively, to maximize the number of perfect matchings, all vertices should be in the largest possible number of short cycles. Therefore, the optimal regular graphs should be bipartite graphs with small girth and the least number of vertices. The existence of $k$-regular bipartite graphs with girth 6 is known for all $k$ (see \cite{Brown} for an example based on finite projective planes). Let $A_{k,6}$ denote the smallest $k$-regular bipartite graph with girth 6, and let $2v_k$ be the number of vertices in such graph. We conclude with the following more general conjecture.

\begin{Conjecture}
	If $A \in \{0, 1\}^{n\times n}$ is $C_4$-free and has at most $kn$ non-zero entries, then 
	$$ \perm(A) \leq \perm(A_{k,6}) ^{n/v_k}.$$
\end{Conjecture}

\end{document}